\documentclass{siamltex}
\usepackage{amsmath,amssymb,graphicx,color}

\def\bfb{\mathbf{b}}

\def\bfs{\mathbf{s}}

\def\bfI{\mathbf{I}}

\def\mE{\mathcal{E}}

\def\mI{\mathcal{I}}

\def\mK{\mathcal{K}}

\def\mN{\mathcal{N}}

\def\mP{\mathcal{P}}

\def\mS{\mathcal{S}}

\def\mZ{\mathcal{Z}}

\newtheorem{remark}[theorem]{Remark}

\title{Mathematical framework for abdominal electrical impedance tomography to assess fatness}

\author{Habib Ammari\thanks{\footnotesize Department of Mathematics,
ETH Z\"urich,
R\"amistrasse 101, CH-8092 Z\"urich, Switzerland (habib.ammari@math.ethz.ch).} \and Hyeuknam Kwon\thanks{\footnotesize Department of Computational
Science and Engineering, Yonsei University, 50 Yonsei-Ro,
Seodaemun-Gu, Seoul 120-749, Korea (hyeuknamkwon@gmail.com, seoj@yonsei.ac.kr).} \and Seungri Lee\footnotemark[2] \and Jin Keun Seo\footnotemark[2]}

\begin{document}

\maketitle

\begin{abstract}
This paper presents a static electrical impedance tomography (EIT) technique that evaluates abdominal obesity by estimating the thickness of subcutaneous fat. EIT has a fundamental drawback for absolute admittivity imaging because of its lack of reference data for handling the forward modeling errors. To reduce the effect of boundary geometry errors in imaging abdominal fat, we develop a depth-based reconstruction method that uses a specially chosen current pattern to construct reference-like data, which are then used to identify the border between subcutaneous fat and muscle. The performance of the proposed method is demonstrated by numerical simulations using 32-channel EIT system and human like domain.
\end{abstract}

\begin{keywords}
 abdominal electrical impedance tomography, reference-like data, outermost region estimation, sensitivity matrix. 
\end{keywords}

\begin{AMS}
  35R30, 49N45, 65N21. 
\end{AMS}

\section{Introduction}\label{sec:intro}
Multi-frequency electrical impedance tomography (mfEIT) \cite{alberti,mfeit3,mfeit2,mfeit1} can be applied to the non-invasive assessment of abdominal obesity, which is a predictor of health risk. MfEIT data of the boundary current-voltage relationship at various frequencies of $<$ 1 MHz reflect the regional distribution of body fat, which is less conductive than water and tissues such as muscle, and can therefore be used to estimate the thicknesses of visceral and subcutaneous adipose tissue. This diagnostic information can be used to assess abdominal obesity, which is considered a cause of metabolic syndrome as well as a risk factor for various other health conditions \cite{Carey:2012,Despres:2006,Razay:2006,Steven:2001,Yusuf:2004}.

The spatial resolution of computed tomography (CT) and magnetic resonance (MR) images is high enough for the assessment of abdominal obesity \cite{Machann:2010,Machann:2005}; however, there are concerns and limitations regarding their use for this purpose; e.g., CT exposes the subject to ionizing radiation, while MR imaging has poor temporal resolution.

Electrical impedance tomography (EIT) is a noninvasive, low-cost imaging technique that provides real-time data without using ionizing radiation. However, experiences over the past three decades have not succeeded in making EIT robust against forward modeling errors such as those associated with uncertainties in boundary geometry and electrode positions. In time-difference EIT, which images changes in the conductivity distribution with time, forward modeling errors are effectively handled and largely eliminated when the data are subtracted from the reference data at a predetermined time. Time-difference EIT for imaging the time changes in the conductivity distribution effectively handle the forward modeling errors, which are somewhat eliminated from the use of time-difference data subtracting reference data at a predetermined fixed time.
 
In static EIT, however, there are no reference data that can be used to eliminate the forward modeling error \cite{liliana}. Creating reference-like data would be the key issue in static EIT.

This paper proposes a new reconstruction method that uses prior anatomical information, at the expense of spatial resolution, to compensate for this fundamental drawback of static EIT and improve its reproducibility. In the case of  abdominal EIT, it is possible to use a spatial prior to handle its inherent ill-posed nature. The proposed method employs a depth-based reconstruction algorithm that takes into account the background region, boundary geometry, electrode configuration, and current patterns. Here, we could take advantage of recent advances in 3D scanner technology to minimize the forward modeling error by extracting accurate boundary geometry and electrode positions.

The proposed method uses a specially chosen current pattern to obtain a depth-dependent data set, generating reference-like data that are used to outline the borders between fat and muscle.
From the relation between a current injected through one pair of electrodes and the induced voltage drop though the other pair of electrodes, we obtain the transadmittance, which is the ratio of the current to the voltage. Hence, the transadmittance depends on the positions of two pairs of electrodes, body geometry, and admittivity distribution. We can extract the corresponding apparent admittivity in term of two pairs of electrodes from the transadmittance divided by a factor involving electrode positions and the body geometry. This apparent admittivity changes with the choice of pairs of electrodes.
(In the special case when the subject is homogeneous, the apparent admittivity does not depend on electrode positions.)
Noting that the change in apparent admittivity in the depth direction can be generated by varying the distance between electrodes, we could probe the admittivity distribution by developing a proper and efficient algorithm based on a least-squares minimization. The performance of the proposed least-squares approach is demonstrated using numerical simulations with a 32-channel EIT system and human like domain.

Future research study is to adopt an mfEIT technique to exploit the frequency-dependent behavior of human tissue \cite{Ammari:2016}. The distribution of visceral fat in abdominal region can then be estimated from the linear relation between the data and the admittivity spectra, and thus obtain a clinically useful absolute conductivity image.

\section{Basic mathematical setting}\label{sec:basic}
\subsection{Model setting}\label{sec:basic-model}
Let an imaging object occupy three (or two) dimensional domain $\Omega$ with its admittivity distribution $\gamma=\sigma+i\omega\epsilon$ where $\sigma$ is the conductivity, $\epsilon$ the permittivity, and $\omega$ the angular frequency.
The domain $\Omega$ can be divide by 4 subregions; subcutaneous fat region $\Omega_f$, muscle region $\Omega_m$, bone region $\Omega_b$, and remaining region $\Omega_r$ as shown in Figure \ref{fig:illustrate-abdomen-tissue-spectroscopy} (a).  
\begin{figure}[htb]
  \begin{center}\begin{tabular}{cc}
	  \includegraphics[height=4.5cm]{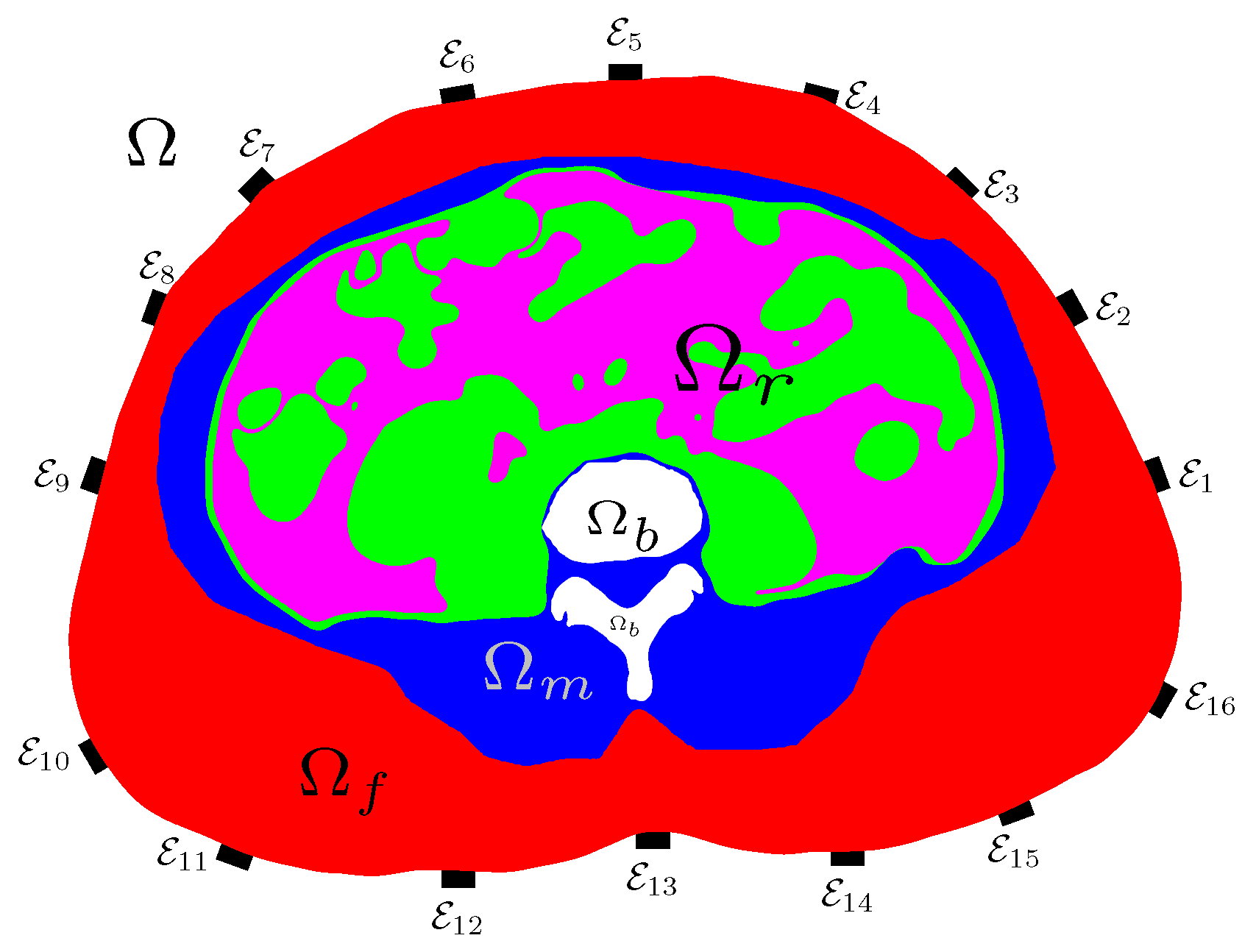} &
	  \includegraphics[height=4.5cm]{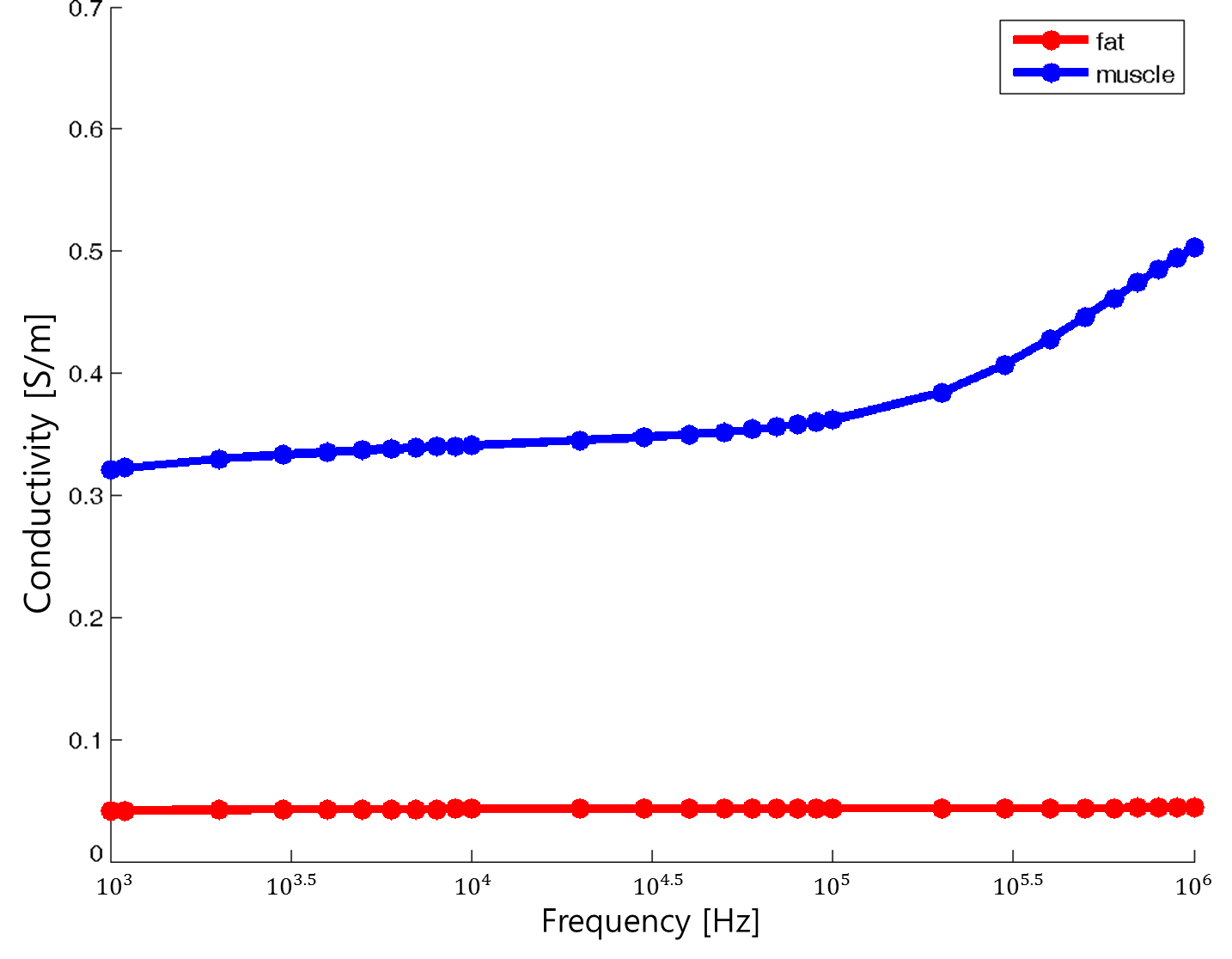} \\
		(a) & (b)
	\end{tabular}\end{center}
	\caption{(a) Simplified abdomen image from CT with $N_E=16$ number of electrodes. Subdomains ($\Omega_f$, $\Omega_m$, $\Omega_b$, $\Omega_r$) are depicted: Red color represents subcutaneous fat region, blue represents muscle region, white represent bone region, pink represents visceral region, and green represents abdominal organs. (b) The electrical conductivity spectroscopy of fat and muscle over frequency range from 1 KHz to 1 MHz.}
	\label{fig:illustrate-abdomen-tissue-spectroscopy}
\end{figure}
The remaining region $\Omega_r$ includes visceral fat, organ, and etc.
The admittivity distribution $\gamma$ in each subregion can be represented as
\begin{equation}
  \gamma(x)=
	\left\{\begin{array}{ll}
	  \gamma_f & \mbox{for}~ x\in\Omega_f , \\
	  \gamma_m & \mbox{for}~ x\in\Omega_m , \\
	  \gamma_b & \mbox{for}~ x\in\Omega_b , \\
		\gamma_r & \mbox{for}~ x\in\Omega_r . 
	\end{array}\right.
\end{equation}

\begin{remark}
The conductivity value of biological tissue is depending on its cell structure and frequency \cite{Ammari:2016}.
The cell structure of fat-tissue and muscle-tissue are quite different \cite{adipeux}.
Hence, the electrical conductivity values of fat and muscle have different aspect over frequency.
As shown in Figure \ref{fig:illustrate-abdomen-tissue-spectroscopy} (b), the spectra of electric conductivity of fat as a function of frequency is different from the one for muscle \cite{Hasgall:2015}.
Since admittivity changes abruptly between subcutaneous fat and muscle, the border between $\Omega_f$ and $\Omega_m$ can be identified by using electrical distribution properties of fat and muscle.
\end{remark}

To inject currents, we attach $N_E$ electrodes $\mE^h_1,\mE^h_2,\cdots,\mE^h_{N_E}$ on the boundary $\partial\Omega$ as shown in Figure \ref{fig:illustrate-abdomen-tissue-spectroscopy} (a).
Here, the superscript $h$ stands for the radius of the circular electrode.
We inject a sinusoidal current of $I$ mA at angular frequency $\omega$ through the pair of electrodes $\mE^h_{k^+}, \mE^h_{k^-}$
for $(k^+,k^-)\in\{(i,j)~:~ i\neq j, ~\mbox{and}~ i,j=1,2,\cdots,N_E\}$. Then the resulting time-harmonic potential $\widetilde{u}^h_k$ is governed by
\begin{equation}
  \left\{\begin{array}{l}
    \nabla \cdot (\gamma \nabla \widetilde{u}^h_k ) = 0 \quad \text{in}~ \Omega , \\
		\left.\widetilde{u}^h_k+z_{k,l}\gamma\frac{\partial \widetilde{u}^h_k}{\partial n}\right|_{\mE^h_{l}} = \widetilde{P}^{h}_{k,l} \quad \mbox{for}~ l=1,\cdots,N_E , \\
		\int_{\mE^h_{k^+}}\gamma\frac{\partial \widetilde{u}^h_k}{\partial n} ~ds = I =-\int_{\mE^h_{k^-}}\gamma\frac{\partial \widetilde{u}^h_k}{\partial n} ~ds , \\
		\int_{\mE^h_{l}}\gamma\frac{\partial \widetilde{u}^h_k}{\partial n} ~ds = 0 \quad\mbox{for}~ l\neq k^\pm ,  \\
		\gamma\frac{\partial \widetilde{u}^h_k}{\partial n} = 0 \quad \mbox{on}~\partial\Omega\setminus\left(\cup_{l=1}^{N_E} \mE^h_{l}\right), 
  \end{array}\right.
	\label{eq:CEM}
\end{equation}
where $n$ is the outward unit normal vector to the boundary $\partial\Omega$, $z_{k,l}$ is a contact impedance, and $\widetilde{P}^h_{k,l}$ is the corresponding constant potential on the electrode $\mE^h_{l}$ for $l=1,2,\cdots,N_E$ \cite{Somersalo:1992} as shown in Figure \ref{fig:CEM-sol-u}. It is worth emphasizing that the solution $\widetilde{u}^h_k$ depends on the electrode radius $h$.
\begin{figure}[htb]
  \begin{center}\begin{tabular}{c}
    \includegraphics[width=12cm]{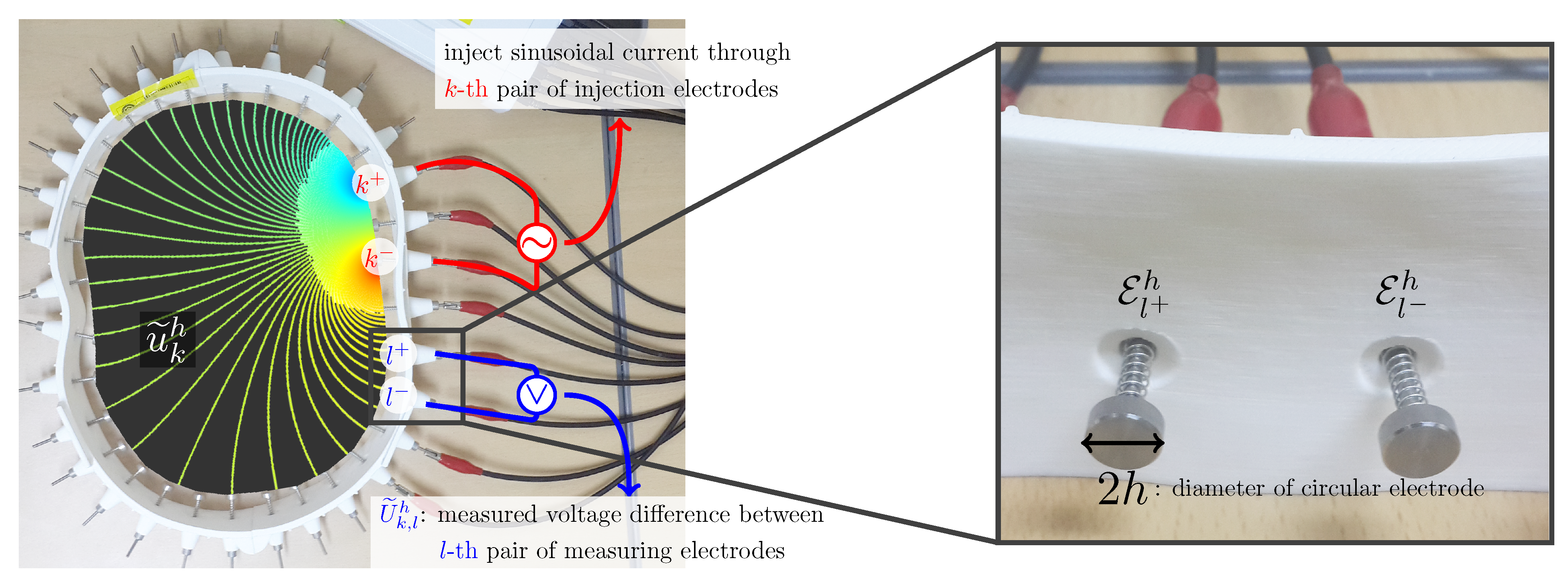} \\
		(a) \hspace{6cm} (b)
	\end{tabular}\end{center}
	\caption{(a) Visualization of induced electric potential $\widetilde{u}^h_k$ with $\gamma=1$ when the current is injected through the pair of electrodes $\mE^h_{k^+}, \mE^h_{k^-}$. (b) The attached face of the used electrodes is a disk with radius $h$.}
	\label{fig:CEM-sol-u}
\end{figure}
Let $\mP$ denote a set of indexes $(k^+,k^-,l^+,l^-)$ indicating inject-measure electrodes, {\it e.g.} $(1,2,3,4)$ means that the electrodes $\mE_1,\mE_2$ are used for current injection and the electrodes $\mE_3,\mE_4$ are used for voltage measurements.
In the static EIT, the following voltage difference data $\widetilde{U}^h_{k,l}$ is used
\begin{equation}
  \widetilde{U}^h_{k,l} := \widetilde{P}^h_{k,l^+}-\widetilde{P}^h_{k,l^-}
  \label{eq:U-CEM}
\end{equation}
for $(k^+,k^-,l^+,l^-)\in\mP$.
In practice, the choice of $(k^+,k^-,l^+,l^-)$ satisfies $l^+\neq k^\pm,~l^-\neq k^\pm$ to avoid unknown contact impedance $z_{k,l^\pm}$ \cite{Paulson:1992}.

Throughout this paper, we assume that the size of the electrodes is very small compared with the size of $\partial\Omega$ so that the solution $u_k$ of the following point electrode model \cite{Hanke:2011} is a good approximation to the solution $\widetilde{u}^h_{k}$ of (\ref{eq:CEM}):
\begin{equation}
  \left\{\begin{array}{ll}
    \nabla \cdot (\gamma \nabla u_{k} ) = 0 & \text{in}~ \Omega ,\\
    \gamma(x)\frac{\partial}{\partial n}u_{k}(x) = I (\delta(x-p_{k^+}) - \delta(x-p_{k^-})) & \text{for}~ x\in\partial\Omega , \\
		\int_{\partial\Omega} u_k ~ds = 0 , 
  \end{array}\right.
	\label{eq:PEM}
\end{equation}
where $p_k$ is the center of the disk $\mE_k^h$ and $\delta$ is the Dirac delta function.
We define the data $U_{k,l}$ by
\begin{equation}
  U_{k,l}
  := u_k(p_{l^+})-u_k(p_{l^-})
	\quad \mbox{for }  (k^+,k^-,l^+,l^-)\in\mP .
  \label{eq:U-PEM}
\end{equation}
We should note that $\widetilde{U}^h_{k,l}\rightarrow U_{k,l}$ as $h\rightarrow0$ except for $k=l$; see Appendix \ref{sec:appendix} (Lemma \ref{lem:CEMisPEM} and Theorem \ref{thm:CEMisPEM}).

\subsection{Layer potentials}
The usage of layer potential enables us to represent the solutions of (\ref{eq:PEM}) and the data (\ref{eq:U-PEM}); see \cite{book,Seo:2012}.
In this section, we introduce layer potentials and represent homogeneous admittivity using the layer potential function.
We consider the three-dimensional case and define the single layer potential $\mS$ by
\begin{equation}
  \mS[f](x):=\int_{\partial\Omega}\Phi(x,y)f(y)ds_y
  \quad\mbox{for}~x\in\Omega, 
\end{equation}
where $\Phi(x,y):=\frac{1}{4\pi|x-y|}$ is the fundamental solution in $\Bbb R^3$.
It is well known that
\begin{equation}
  \lim_{t\rightarrow0^+}\langle n(x),\nabla\mS[f](x\pm tn(x))\rangle =(\pm\frac{1}{2}\mI+\mK^*)[f](x)
  \quad(x\in\partial\Omega),
\end{equation}
where $\mI$ is the identity operator, $\mK^*$ is the dual operator of $\mK$ and $\mK:L^p(\partial\Omega)\rightarrow L^p(\partial\Omega)$, $1<p<\infty$, is the trace operator defined by
\begin{equation}
  \mK[f](x) = \int_{\partial\Omega}\frac{\langle y-x,n(y)\rangle}{4\pi|x-y|^3}f(y) ds_y
  \quad(x\in\partial\Omega).
\end{equation}

For $x\in\Omega$, we define the Neumann function $\mN_\gamma(x,y)$ by
\begin{equation}
  \left\{\begin{array}{rcll}
    -\nabla_y\cdot(\gamma(y)\nabla_y \mN_\gamma(x,y)) &=& \delta(x-y) & \mbox{for}~ y\in\Omega ,\\
	  -\gamma(y)\frac{\partial}{\partial n_y}\mN_\gamma(x,y) &=& \frac{1}{|\partial \Omega|} & \mbox{for}~ y\in\partial\Omega ,\\
	  \int_{\partial \Omega} \mN_\gamma(x,y) ds_y &=& 0.&
  \end{array}\right.
	\label{eq:def-Neumann}
\end{equation}
In the case when $\gamma=1$, the Neumann function $\mN_1(x,y)$ can be, up to a constant,  represented by \cite{book}
\begin{equation}
  \mN_1(x,y)=\left(-\frac{1}{2}\mI+\mK\right)^{-1}[\Phi(x,\cdot)](y)
\end{equation}
for $x\in\partial\Omega$.
Then, the constant admittivity in $\Omega$ can be computed using the Neumann function.
\begin{lemma}\label{lem:const}
If $\gamma$ is constant in $\Omega$, then it can be expressed as
\begin{equation}
  \gamma=\frac{I}{U_{k,l}}\left(\mN_1(p_{k^+},p_{l^+})-\mN_1(p_{k^+},p_{l^-})-\mN_1(p_{k^-},p_{l^+})+\mN_1(p_{k^-},p_{l^-})\right) .
\end{equation}
\end{lemma}
\begin{proof}
The solution $u_k$ in (\ref{eq:PEM}) can be represented using the Neumann function by
\begin{eqnarray}
  u_k(x)
	&=& \int_{\partial\Omega}{\mN_1(x,y) \frac{\partial}{\partial n} u_k(y) ds_{y}} \label{eq:neumann}\\
	&=& \frac{I}{\gamma}(\mN_1(x,p_{k^+}) - \mN_1(x,p_{k^-})) . \nonumber
\end{eqnarray}
Then, the voltage difference $U_{k,l}=u_k(p_{l^+})-u_k(p_{l^-})$ is given by
\begin{equation}
  U_{k,l} = \frac{I}{\gamma} \left( (\mN_1(p_{k^+},p_{l^+})-\mN_1(p_{k^+},p_{l^-}))-(\mN_1(p_{k^-},p_{l^+})-\mN_1(p_{k^-},p_{l^-})) \right).
\end{equation}
Hence, we have
\begin{equation}
  \gamma = \frac{I}{U_{k,l}}\left((\mN_1(p_{k^+},p_{l^+})-\mN_1(p_{k^+},p_{l^-}))-(\mN_1(p_{k^-},p_{l^+})-\mN_1(p_{k^-},p_{l^-}))\right),
\end{equation}
which is the desired result. 
\end{proof}

\subsection{Geometric factor in the data}\label{sec:transimpedance}
In this section, the relation between the data and admittivity distribution is introduced.
We first consider the two dimensional case and assume that $\Omega$ is a disk with radius $s$ and a homogeneous admittivity $\gamma$.
\begin{lemma}\label{lem:homg-disk}
Let $p_{k^+},p_{k^-},p_{l^+},p_{l^-}$ be the positions of the electrodes on the boundary $\partial \Omega$ as shown in Figure \ref{fig:U-disk}.
If the admittivity $\gamma$ is constant in the disk $\Omega$ with radius $s$, then
\begin{equation}
  \gamma
	=
	\frac{1}{U_{k,l}}\frac{I}{\pi}\log\left(\frac{\sin(\frac{d_1+d_2+d_3}{2s})~\sin(\frac{d_2}{2s})}{\sin(\frac{d_1+d_2}{2s})~\sin(\frac{d_2+d_3}{2s})}\right),
	\label{eq:gamma=U-U}
\end{equation}
where $d_1,d_2,d_3$ are the arc lengths between adjacent $p_{k^+},p_{k^-},p_{l^+},p_{l^-}$ and $U_{k,l}$ is defined by (\ref{eq:U-PEM}).
\end{lemma}
\begin{figure}[htb]
  \begin{center}\begin{tabular}{cc}
  	\begin{minipage}{2.2cm}\includegraphics[height=2.4cm]{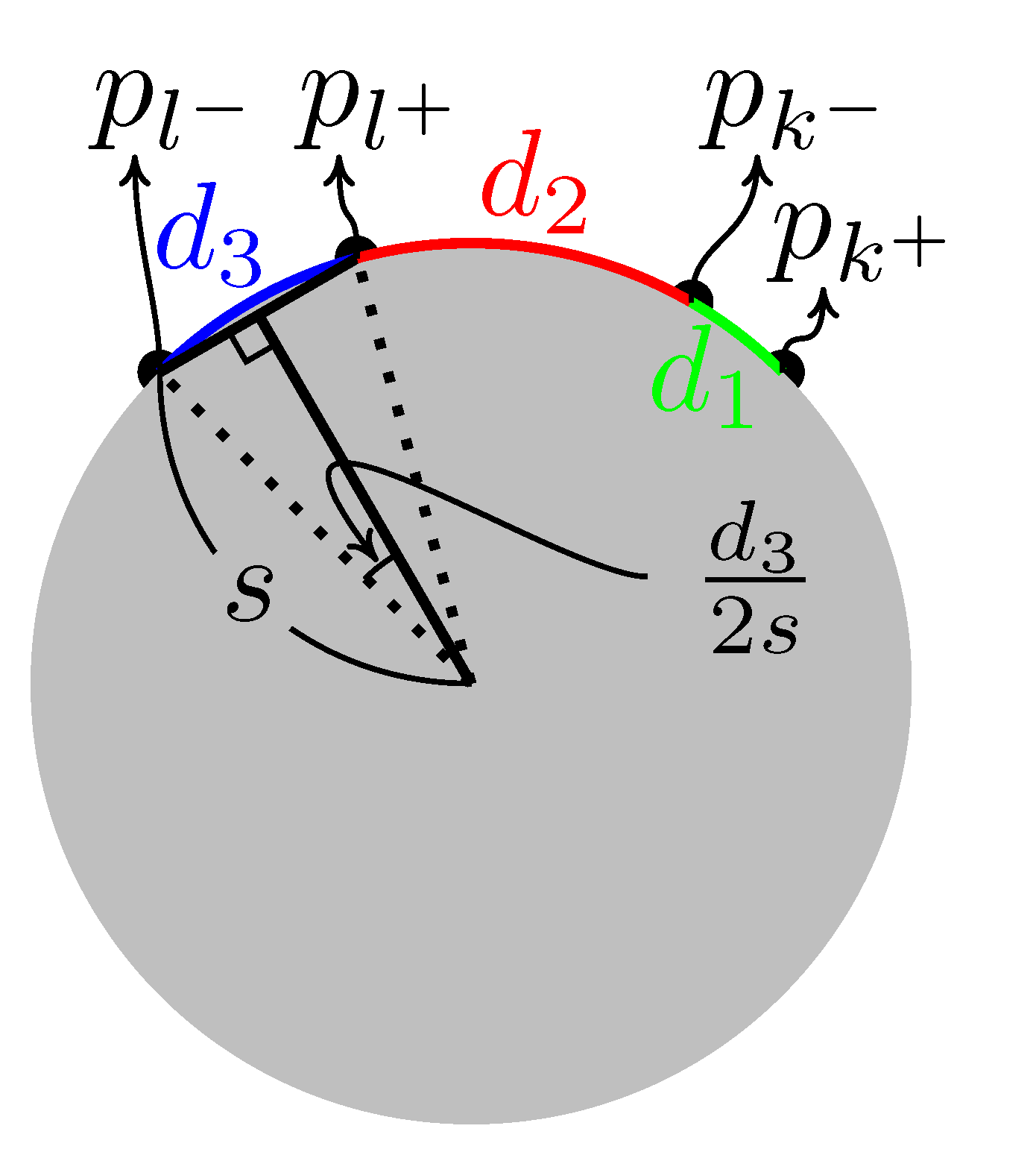}\end{minipage}
  	&
  	\begin{minipage}{9.8cm}\begin{tabular}{c|ccc}
		  ~&
	    \includegraphics[height=2.4cm]{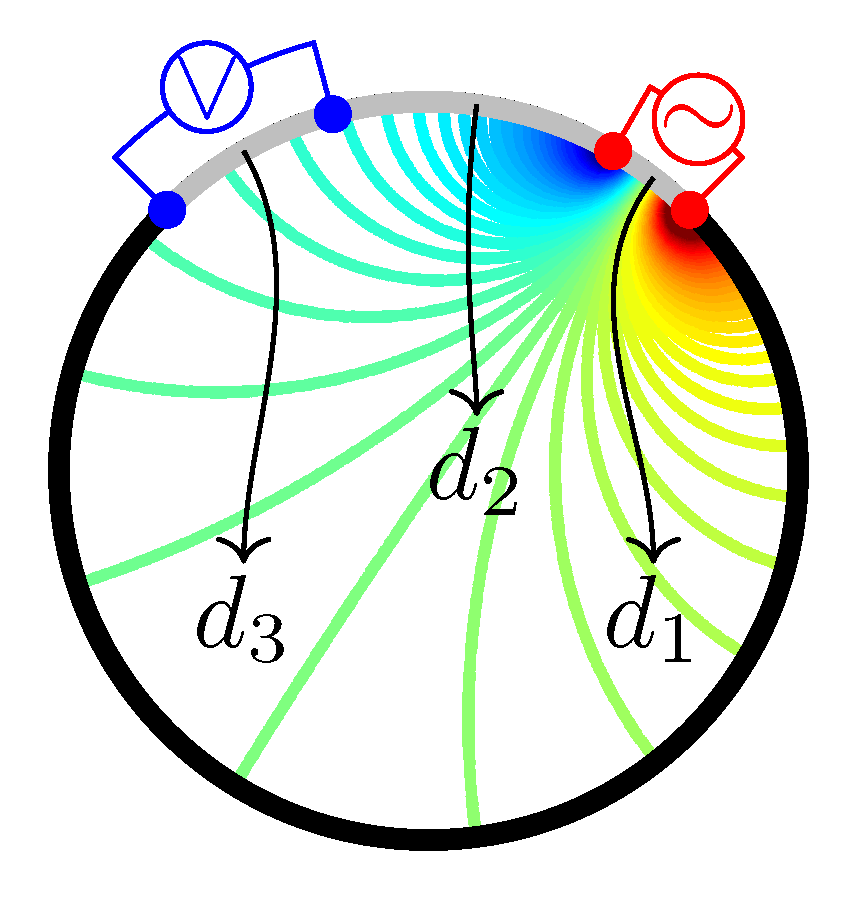} &
	    \includegraphics[height=2.4cm]{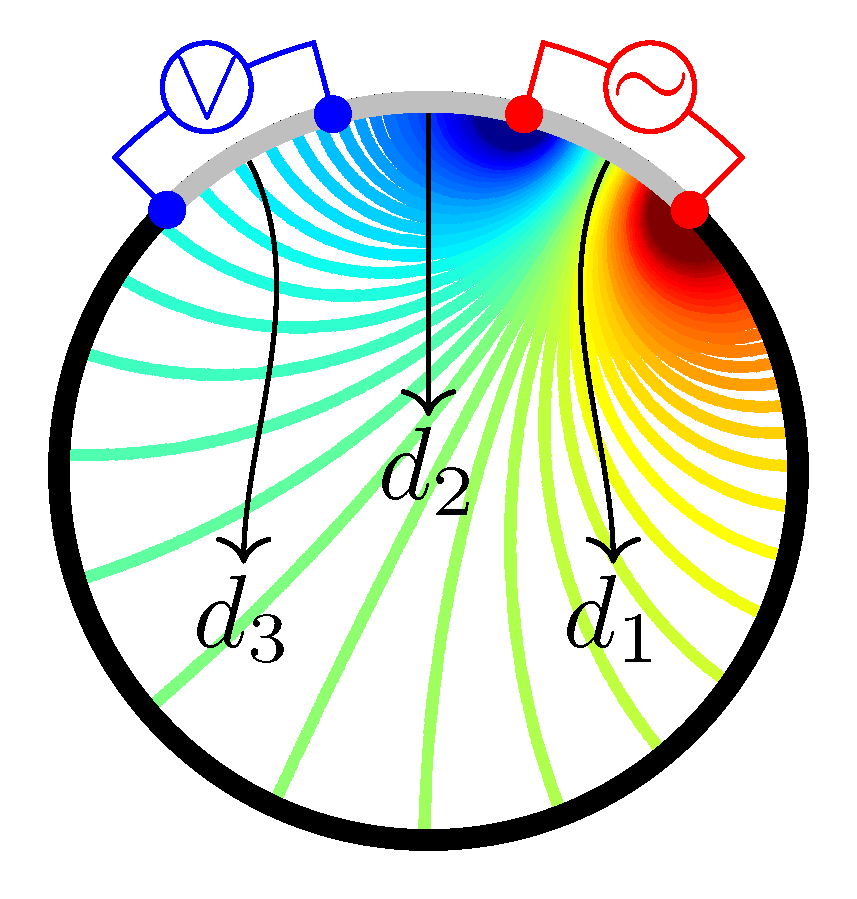} &
	    \includegraphics[height=2.4cm]{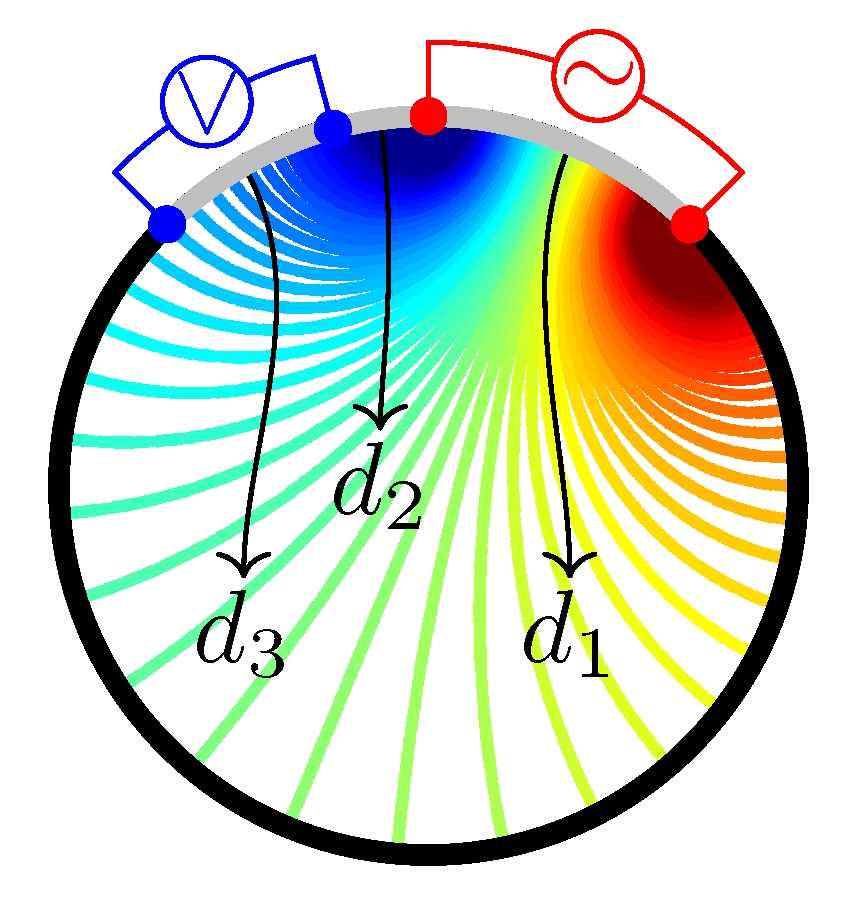} \\ \hline
		  {\scriptsize $U_{k,l}$} & 0.0125 & 0.0331 & 0.0774 \\
		  {\scriptsize $G_{k,l}$} & 3.0000 & 3.0000 & 3.0000
	  \end{tabular}\end{minipage}
	  \\
	  (a) & (b)
	\end{tabular}\end{center}
	\caption{(a) Illustration of the notation used in Lemma \ref{lem:homg-disk}: electrode positions $p_{k^+}$, $p_{k^-}$, $p_{l^+}$, $p_{l^-}$, arc length $d_1$, $d_2$, $d_3$, and radius $s$ of the disk $\Omega$. (b) For $\gamma=3$, the computed $U_{k,l}$ and $G_{k,l}$ are represented with various choices of $p_{k^+},p_{k^-},p_{l^+},p_{l^-}$.}
	\label{fig:U-disk}
\end{figure}
\begin{proof}
For $x,y\in\partial\Omega$, the Neumann function $\mN_1(x,y)=2\Phi(x,y)$, where $$\Phi(x,y) := -\frac{1}{2\pi} \log|x-y|$$ is the fundamental solution of the Laplace equation in two dimensions.  
By Lemma \ref{lem:const}, the voltage $U_{k,l}$ can be calculated by
\begin{eqnarray*}
  U_{k,l}
	&=&
	\frac{I}{\gamma} \left( (\mathcal{N}_1(p_{k^+},p_{l^+}) - \mathcal{N}_1(p_{k^+},p_{l^-})) - (\mathcal{N}_1(p_{k^-},p_{l^+}) - \mathcal{N}_1(p_{k^-},p_{l^-})) \right)
	\\
	&=&
	\frac{I}{\gamma\pi} \log\left(\frac{|p_{k^+}-p_{l^-}|~|p_{k^-}-p_{l^+}|}{|p_{k^+}-p_{l^+}|~|p_{k^-}-p_{l^-}|}\right)
	\\
	&=&
	\frac{I}{\gamma\pi} \log\left(\frac{2\sin(\frac{3d}{2s})~2\sin(\frac{d}{2s})}{2\sin(\frac{2d}{2s})~2\sin(\frac{2d}{2s})}\right).
\end{eqnarray*}
Hence, we get (\ref{eq:gamma=U-U}) as desired.
\end{proof}

To get the admittivity value $\gamma$ from the data $U_{k,l}$, we need to multiply by 
$\frac{I}{\pi}\log\left(\frac{\sin(\frac{d_1+d_2+d_3}{2s})~\sin(\frac{d_2}{2s})}{\sin(\frac{d_1+d_2}{2s})~\sin(\frac{d_2+d_3}{2s})}\right)$,
which contains a geometric factor (the radius $s$) of the domain $\Omega$.
This means that the data $U_{k,l}$ involves a geometric factor and so, in order to get the admittivity $\gamma$ we have to consider the geometry of $\Omega$ and electrode configurations.
As Figure \ref{fig:U-disk} (b) shows, the data $U_{k,l}$ is changing when the used electrode pair for injecting current are changing although $\gamma$ is kept constant ($\gamma=3$). In the 
following, a new combination of the data, $G_{k,l}$, is introduced.  $G_{k,l}$ is designed to not change over geometric factors, the geometry of $\Omega$ nor the electrode configurations.

Now we consider the arbitrary simple closed domain $\Omega$ with a homogeneous admittivity distribution $\gamma$.
The voltage difference $U_{k,l}$ and the unknown admittivity $\gamma$ to be imaged are related by
\begin{equation}
  IU_{k,l}=\int_\Omega \gamma\nabla u_k\cdot\nabla u_l ~dx .
	\label{eq:transZ}
\end{equation}
Since $\gamma$ is a constant in $\Omega$, we have
\begin{equation}
  \gamma
	~=~
	\frac{1}{IU_{k,l}}\int_\Omega \nabla v_k\cdot\nabla v_l ~dx
	~=~
	\frac{V_{k,l}}{U_{k,l}} , 
	\label{eq:V/U}
\end{equation}
where $v_k$ is the solution of (\ref{eq:PEM}) with $\gamma=1$ and $V_{k,l}:=v_k(p_{l^+})-v_k(p_{l^-})$.
Hence, the ratio between $U_{k,l}$ and $V_{k,l}$ gives the homogeneous value $\gamma$ by eliminating geometric influence.
Note that $V_{k,l}$ can be obtained by numerical simulation with known geometry information.
Now, we can define the geometry factor canceled data $G_{k,l}$, named geometry-free data, by the data ratio of {simulated} data $V_{k,l}$ and measured data $U_{k,l}$
\begin{equation}
  G_{k,l}:=\frac{V_{k,l}}{U_{k,l}} .
	\label{eq:def-V/U}
\end{equation}

For inhomogeneous $\gamma$, (\ref{eq:def-V/U}) gives some kind of average value of the admittivity in $\Omega$.
The sense of average looks similar as a weighted harmonic average with the weighting factor which depends on the position $x\in\Omega$ and the positions of the used electrodes $(p_{k^+},p_{k^-},p_{l^+},p_{l^-})$ for injecting currents and measuring voltages.
In dimension one, let $\Omega:=(0,1)$ and admittivity $\gamma(x)=\gamma_1$ for $x\in(0,a)$ and $\gamma(x)=\gamma_2$ for $x\in(a,1)$.
Then $U_{k,k}=u_k(1)-u_k(0)$ for $p_{k^+}=1$ and $p_{k^-}=0$ where $u_k$ satisfies $(\gamma u_k')'=0$ in $\Omega$ and $\gamma(x) u_k'(x)=I$ for $x\in\{0,1\}$.
By integration, we can obtain $u(x)=I\int 1/\gamma(x) dx$. Hence
\begin{equation}
  G_{k,k}=\frac{V_{k,k}}{U_{k,k}}=\frac{I\int_0^1 1 ~dx}{I\int_0^1 1/\gamma(x) ~dx}=\frac{1}{a/\gamma_1+(1-a)/\gamma_2} .
\end{equation}
Therefore, in one dimension, (\ref{eq:def-V/U}) gives the harmonic average value of inhomogeneous admittivity distribution in $\Omega$ with volume ratio as a weighting factor.
Note that there are three unknowns, $a$, $\gamma_1$, and $\gamma_2$, in (\ref{eq:def-V/U}).

We can generalize the result in one dimension to the $n$-dimensional case.
In $\mathbb{R}^n$, by using (\ref{eq:transZ}), equation (\ref{eq:def-V/U}) can be written as
\begin{equation}
  G_{k,l}=\frac{V_{k,l}}{U_{k,l}}
	=
	\frac{I\int_\Omega (\nabla v_k/I)\cdot(\nabla v_l/I) ~dx}{I\int_\Omega (1/\gamma)~(\gamma\nabla u_k/I)\cdot(\gamma\nabla u_l/I) ~dx} .
\end{equation}
Let $w_{k,l}(x)$ be the weight function defined by $$w_{k,l}(x):=(\gamma(x)\nabla u_k(x)/I)\cdot(\gamma(x)\nabla u_l(x)/I)$$ for $x\in\Omega$. Then, 
\begin{equation}
\displaystyle  G_{k,l}
	=
	\frac{\displaystyle \int_\Omega w_{k,l}(x) ~dx}{\displaystyle\int_\Omega w_{k,l}(x)/\gamma(x) ~dx}
	~\frac{\displaystyle\int_\Omega \widetilde{w}_{k,l}(x) ~dx}{\displaystyle\int_\Omega w_{k,l}(x) ~dx} , 
\end{equation}
where $\widetilde{w}_{k,l}(x):=(\nabla v_k(x)/I)\cdot(\nabla v_l(x)/I)$ for $x\in\Omega$.
For being a weighted harmonic average, one should have $\int_\Omega \widetilde{w}_{k,l}(x) ~dx=\int_\Omega w_{k,l}(x) ~dx$.
However, in general $\int_\Omega \widetilde{w}_{k,l}(x) ~dx \neq \int_\Omega w_{k,l}(x) ~dx$. Note also that $\widetilde{w}_{k,l}(x)$ can be computed numerically while $w_{k,l}(x)$ can not be because of the unknown values of $\gamma(x)\nabla u_k(x)$ and $\gamma(x)\nabla u_l(x)$ for $x\in \Omega$.

Since there could be forward modeling errors in computing $V_{k,l}$, the geometry-free data $G_{k,l}$ might be corrupted by such errors.
Therefore, we use the difference of $G_{k,l}^{-1}$ to compensate forward modeling errors in $v_k$.
We define the inverse of difference of geometry-free data $B_{k,l,k',l'}$ such that
\begin{equation}
  B_{k,l,k',l'}
	:=
	G_{k,l}^{-1}-G_{k',l'}^{-1}.
\end{equation}
As it will be shown in Theorem \ref{eq:main}, the use of $B_{k,l,k',l'}$ allows to eliminate the effect of homogeneous admittivity near the boundary.

\subsection{Decay estimation of the electric current}\label{sec:decay}
In this section, we investigate the weight function $w_{k,l}$.
For the aim of clarity, we use the point-electrode model (\ref{eq:PEM}) which is a reasonably good approximation of (\ref{eq:CEM}).
Then, by Theorem \ref{thm:CEMisPEM} the corresponding voltage difference data $\widetilde{U}^h_{k,l}$ can be approximated by $U_{k,l}$:
\begin{eqnarray*}
  U_{k,l}
	&=& u_k(p_{l^+})-u_k(p_{l^-}) \\
  &=& I((\mN_\gamma(p_{l^+},p_{k^+})-\mN_\gamma(p_{l^+},p_{k^-}))-(\mN_\gamma(p_{l^-},p_{k^+})-\mN_\gamma(p_{l^-},p_{k^-}))),
\end{eqnarray*}
where $\mN_\gamma$ is the Neumann function defined in (\ref{eq:def-Neumann}).
Moreover, the voltage difference $U_{k,l}$ and the unknown admittivity to be imaged are connected by the following relation: 
\begin{eqnarray*}
  IU_{k,l}
	&=&
	\int_\Omega \gamma \nabla u_{k}\cdot\nabla u_{l} dx
	\\
  &=&
	\int_\Omega\gamma(x)\nabla\left(\mN_\gamma(x,p_{k^+})-\mN_\gamma(x,p_{k^-})\right)\cdot\nabla\left(\mN_\gamma(p_{l^+},x)-\mN_\gamma(p_{l^-},x)\right) dx .
\end{eqnarray*}
In particular, when the admittivity is a constant $\gamma=\gamma_0$ and $\Omega$ is the half space in 
$\mathbb{R}^3$,
\begin{eqnarray*}
  U_{k,l}
  &=&
	\frac{1}{4\pi^2}\int_\Omega \gamma_0\left(\frac{x-p_{k^+}}{|x-p_{k^+}|^3}-\frac{x-p_{k^-}}{|x-p_{k^-}|^3}\right)\cdot\left(\frac{x-p_{l^+}}{|x-p_{l^+}|^3}-\frac{x-p_{l^-}}{|x-p_{l^-}|^3}\right) dx
\end{eqnarray*}
for $I=1$.
According to the decay behavior of the Neumann function \cite{Widman:1982}, there exists a positive constant $C$ such that
\begin{equation}
  \left|\nabla (\mN_\gamma(x,p_{k^+})-\mN_\gamma(x,p_{k^-}))\right| \leq C \frac{d_k}{|x-\widetilde{p}_{k}|^3}
  \quad\mbox{if}~\mbox{dist}(x,\partial\Omega)> d_k ,
\end{equation}
where $d_k=|p_{k^+}-p_{k^-}|$, $\widetilde{p}_{k}$ is a point in the line-segment joining $p_{k^+}$ and $p_{k^-}$.
Therefore,
\begin{equation}
  |\nabla u_{k} \cdot \nabla u_{l}|\leq C \frac{d_kd_l}{|x-\widetilde{p}_{k}|^3|x-\widetilde{p}_{l}|^3}
  \quad\mbox{if}~\mbox{dist}(x,\partial\Omega) > \max\{d_k,d_l\}
\end{equation}
for some positive constant $C$.
Hence, under the assumption of transversally constant admittivity, the measured voltage is very little affected by the admittivity far from the electrodes $p_{k^+},p_{k^-},p_{l^+},p_{l^-}$.

\section{Estimation of outermost region}\label{sec:sfat-homg}
In this section, we consider a simple model having a homogeneous admittivity distribution in the outermost region, {\it i.e.}, $\nabla\gamma(x)=0$ for $x\in\Omega_f$.
Under this assumption, the border between $\Omega_f$ and $\Omega_m$ will be identified.
Based on the mathematical analysis performed in section \ref{sec:main-homg-sfat}, a reconstruction algorithm is provided in section \ref{sec:recon-homg-sfat}.

\subsection{Mathematical analysis}\label{sec:main-homg-sfat}
Based on the observation in Lemma \ref{lem:const}, the constant admittivity near the boundary can be removed.
Indeed, the following theorem says that in the geometry-free data $G_{k,l}$, we can eliminate the influence of homogeneous admittivity $\gamma$ in the outermost region {by introducing a reference-like data $G_{1,2}$}.
\begin{theorem}\label{thm:main}
Let $\Omega_f$ be a subdomain in $\Omega$ satisfying $\partial\Omega\subset\overline{\Omega_f}$.
Assume that $\gamma$ is a constant $\gamma=\gamma_0$ in $\Omega_f$.
Then the data $B_{k,l,1,2}=G_{k,l}^{-1}-G_{1,2}^{-1}$ can be expressed as
\begin{equation}
  B_{k,l,1,2}
	=
  -\int_{\Omega\setminus\Omega_f}\frac{\nabla v_l\cdot\beta\nabla u_k}{V_{k,l}}-\frac{\nabla v_2\cdot\beta\nabla u_1}{V_{1,2}} ~dy,
	\label{eq:main}
\end{equation}
where $\beta(y)=\left(\gamma(y)/\gamma_0-1\right)/I$.
\end{theorem}
\begin{proof}
For the solution $u_k$ of (\ref{eq:PEM}), we have
\begin{eqnarray*}
  u_k(x)
  &=&
  \int_{\Omega}{\nabla\mN_1(x,y) \cdot \nabla u_k(y) dy} \\
  &=&
  \int_{\Omega} \nabla\mN_1(x,y) \cdot\left(\frac{1}{\gamma_0}\gamma(y)\nabla u_k(y)\right) dy
  -\int_{\Omega} \nabla\mN_1(x,y) \cdot\left(\beta(y)\nabla u_k(y)\right) dy \\
  &=&
  \frac{I}{\gamma_0}\left(\mN_1(x,p_{k^+})-\mN_1(x,p_{k^-})\right)
  -
  \int_{\Omega\setminus\Omega_f} \nabla\mN_1(x,y)\cdot\left(\beta(y)\nabla u_k(y)\right) dy.
\end{eqnarray*}
Then, the potential difference $U_{k,l}$ is given by
\begin{eqnarray*}
  U_{k,l}
	=
	\frac{I}{\gamma_0}
	\left(\left(\mN_1(p_{l^+},p_{k^+})-\mN_1(p_{l^+},p_{k^-})\right)
	-
	\left(\mN_1(p_{l^-},p_{k^+})-\mN_1(p_{l^-},p_{k^-})\right)\right)
	\\
	-
	\int_{\Omega\setminus\Omega_f} \left(\nabla \mN_1(p_{l^+},y)-\nabla \mN_1(p_{l^-},y)\right)\cdot\left(\beta(y)\nabla u_k(y)\right) dy.
\end{eqnarray*}
With the help of $$\alpha_{k,l}:=\frac{(\mN_1(p_{l^+},p_{k^+})-\mN_1(p_{l^+},p_{k^-}))-(\mN_1(p_{l^-},p_{k^+})-\mN_1(p_{l^-},p_{k^-}))}{(\mN_1(p_{2^+},p_{1^+})-\mN_1(p_{2^+},p_{1^-}))-(\mN_1(p_{2^-},p_{1^+})-\mN_1(p_{2^-},p_{1^-}))},$$ we can eliminate the first term as follows:
\begin{eqnarray}
  U_{k,l}
  -
  \alpha_{k,l}
  U_{1,2}
  &=&
  -\int_{\Omega\setminus\Omega_f}\nabla\left(\mN_1(p_{l^+},y)-\mN_1(p_{l^-},y)\right)\cdot(\beta(y)\nabla u_k(y)) \nonumber
  \\&&
  ~\quad~-\alpha_{k,l}\nabla\left(\mN_1(p_{2^+},y)-\mN_1(p_{2^-},y)\right)\cdot(\beta(y)\nabla u_1(y)) dy. 
  \label{eq:UformulaD}
\end{eqnarray}
Since $\mN_1(p_{l^+},y)-\mN_1(p_{l^-},y)=v_l(y)/I$, $\alpha_{k,l}=V_{k,l}/V_{1,2}$ and (\ref{eq:main}) follows by dividing (\ref{eq:UformulaD}) by $V_{1,2}$.
\end{proof}

Theorem \ref{thm:main} is established for the data $U_{k,l}$ using the point electrode model (\ref{eq:PEM}) while the measurable data $\widetilde{U}^h_{k,l}$ follows from the complete electrode model (\ref{eq:CEM}).
In order to apply Theorem \ref{thm:main} to the measured data $\widetilde{U}^h_{k,l}$, we need Theorem \ref{thm:CEMisPEM}.
Since we do not use the data where we inject currents, Theorem \ref{thm:CEMisPEM} can be applied to the measured data $\widetilde{U}^h_{k,l}$ with $l^+ \neq k^\pm$ and $l^- \neq k^\pm$ so that we can introduce a similar argument as the one in Theorem \ref{thm:main}.
By the triangle inequality,
\begin{equation}
  \|\widetilde{u}^h_k-u_k\| \le \|(\widetilde{u}^h_k-\hat{u}^h_k) 
  - (u_k-\hat{v}_k)\|+\|\hat{u}^h_k-\hat{v}_k\|, 
\end{equation}
where $\hat{u}^h_k$ and $\hat{v}_k$ are the solution of (\ref{eq:CEM}) and (\ref{eq:PEM}) with $\gamma=1$ respectively.
By \cite{Hanke:2011} and Theorem \ref{thm:CEMisPEM}, we have
\begin{equation}
  \|\widetilde{u}^h_k-u_k\| \le Ch .
\end{equation}
Hence, we can apply Theorem \ref{thm:main} to measured data $\widetilde{U}^h_{k,l}$ and $\widetilde{u}^h_k$ to obtain
\begin{eqnarray}
  \widetilde{U}^h_{k,l}
  -
  \alpha_{k,l}
  \widetilde{U}^h_{1,2}
  &=&
  -\int_{\Omega\setminus\Omega_f}\nabla\left(\mN(p_{l^+},y)-\mN(p_{l^-},y)\right)\cdot(\beta(y)\nabla \widetilde{u}^h_k(y)) \nonumber
  \\&&
  ~~-\alpha_{k,l}\nabla\left(\mN(p_{2^+},y)-\mN(p_{2^-},y)\right)\cdot(\beta(y)\nabla \widetilde{u}^h_1(y)) dy \nonumber
  \\&&
  +o(h) . \label{eq:main2}
\end{eqnarray}
Expansion (\ref{eq:main2}) shows that, up to $o(h)$, only points in $\Omega \setminus \Omega_f$ contribute to the weighted measurements $\widetilde{U}^h_{k,l}-\alpha_{k,l}\widetilde{U}^h_{1,2}$. It plays a key role in developing an efficient algorithm for reconstructing the internal boundary of $\Omega_f$ from current-voltage boundary measurements without any reference data.

\subsection{Reconstruction method}\label{sec:recon-homg-sfat}
In this section, we introduce an outermost-region subtraction method for identifying the inhomogeneous border near boundary based on Theorem \ref{thm:main}.
Specifically, the inhomogeneous border near the $n$-th electrode will be considered.
Then we define the set $\mP_n$ by collecting all possible 4-component combinations (unlike permutations) from $\mZ_n:=\{n-3,n-2,\cdots,n+4\}$ such that the order of selection of the first 2-components and the last 2-components does not matter:
\begin{equation}
  \mP_n
	:=
	\{(i,j,k,l)~:~ i<j,~ k<l,~ i<k,~j\neq l,~~ i,j,k,l\in\mZ_n\}.
	\label{eq:mPn}
\end{equation}
For the data $U^h_{1,2}$, we define $1^+=n,~1^-=n+1,~2^+=n-1,~2^-=n+2$ as shown in Figure \ref{fig:mPs}.
\begin{figure}[htb!]
  \centering \includegraphics[width=10cm]{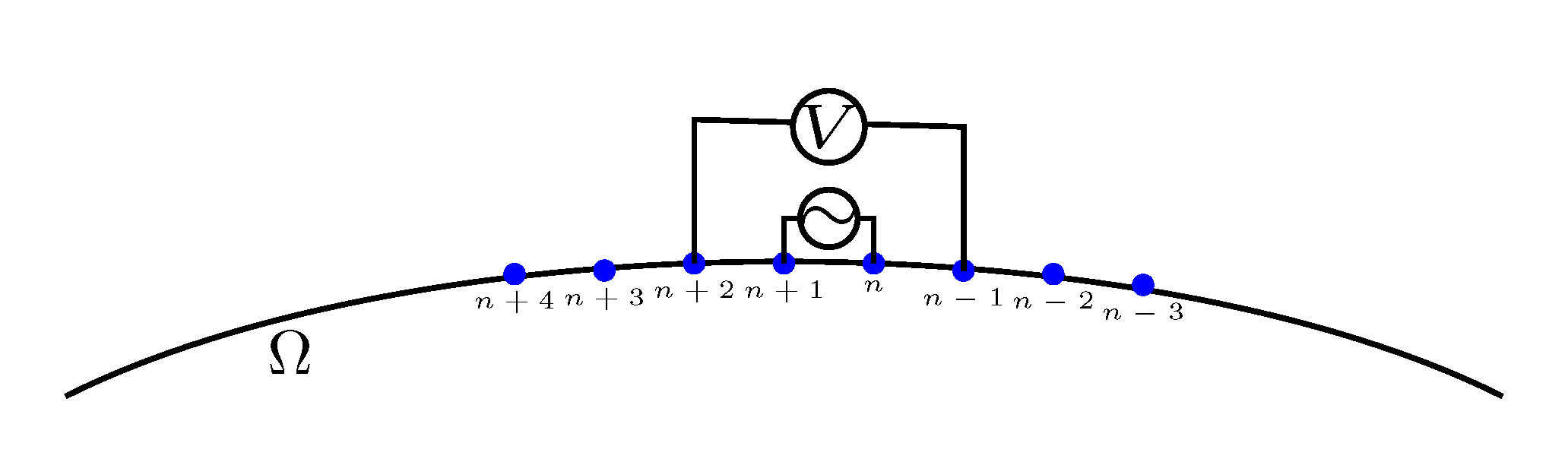}
	\caption{Illustration of choosing the data $U^h_{1,2}$ when we are interested in the inhomogeneous border near $n$-th electrode.}
	\label{fig:mPs}
\end{figure}

We denote the data vector as $\bfb:=\left[B_{k,l,1,2}\right]_{(k^+,k^-,l^+,l^-)\in\mP_n}$ for reconstruction and assume that $\gamma \approx \gamma_0$ to write that $u_k \approx (1/\gamma_0) v_k$.
Note that the data $\bfb$ is rarely affected by the admittivity distribution far from $p_{k^\pm}$, $p_{l^\pm}$ as explained in section \ref{sec:decay}.
Hence, we restrict our reconstruction domain to $D\subset\Omega$ near the electrodes $\{\mE_{n-3},\mE_{n-2},\cdots,\mE_{n+3},\mE_{n+4}\}$.
Figure \ref{fig:recon-mesh} shows an example of $D$ with discretization near eight electrodes $\{\mE_1,\mE_2,\mE_3,\mE_4,\mE_5,\mE_6,\mE_7,\mE_8\}$ for $n=4$.
\begin{figure}[htb!]
	\centering \includegraphics[width=8cm]{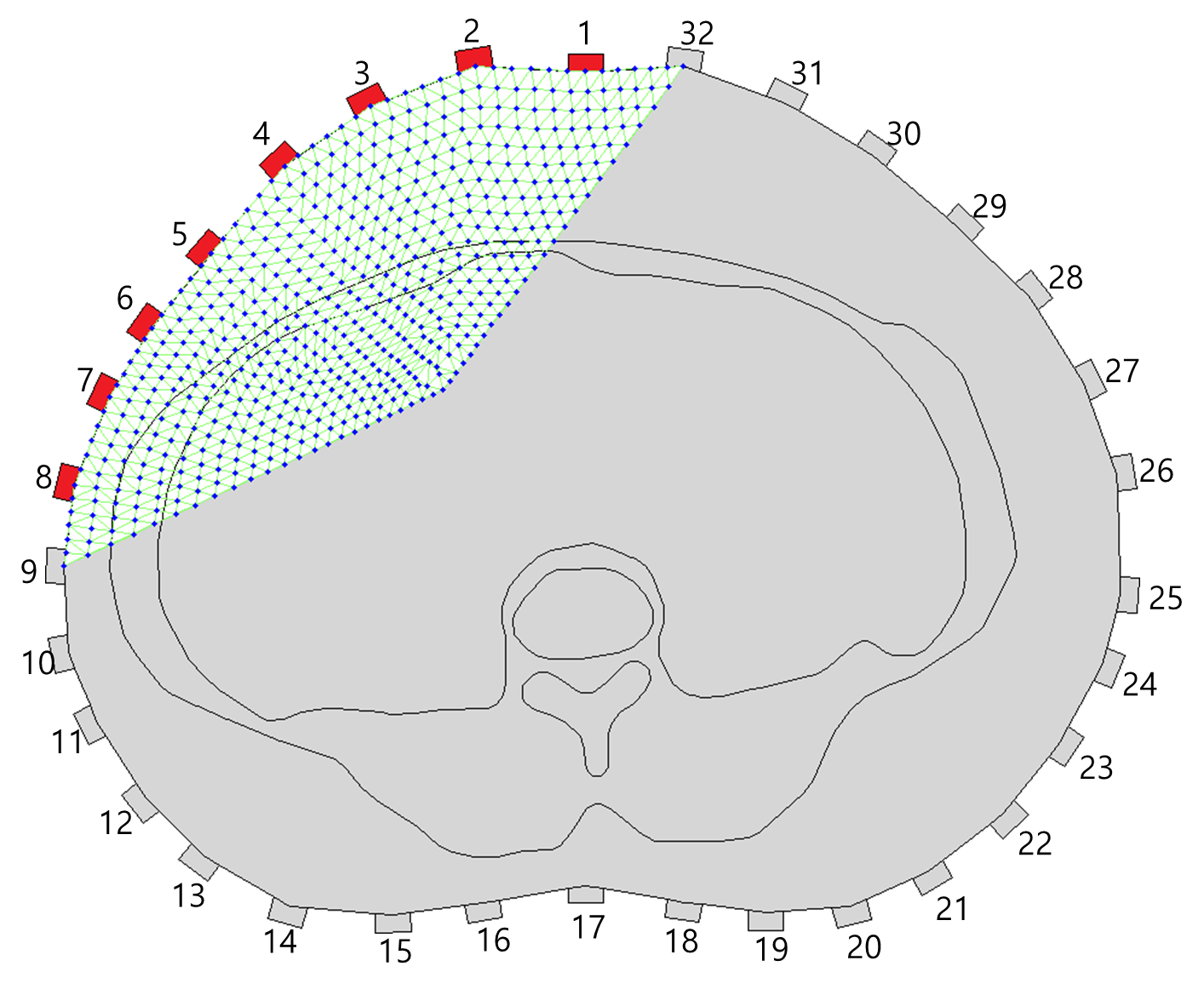}
	\caption{Illustration of $D$ with discretization of using eight electrodes $\{\mE_1,\mE_2,\mE_3,\mE_4,\mE_5,\mE_6,\mE_7,\mE_8\}$}
	\label{fig:recon-mesh}
\end{figure}

We discretize the imaging domain $D$ into $N_T$ elements $D=\cup_{m=1}^{N_T}D_m$ with assuming that the admittivity $\gamma$ is isotropic and constant on each $D_m$ for $m=1,2,\cdots,N_T$.
Let $\kappa$ be the vector $\kappa=[\kappa_1,~\kappa_2,~\cdots,\kappa_{N_T}]$ with
\begin{equation}
  \kappa_m := -\frac{1}{I}\left(\frac{\left.\gamma\right|_{D_m}}{\gamma_0^2}-\frac{1}{\gamma_0}\right)
\end{equation}
for $m=1,2,\cdots,N_T$.
Then we linearize (\ref{eq:main}). By using $u_k \approx (1/\gamma_0) v_k$ in $D_m$,  the following linear system for outermost-region subtracted image $\kappa$ can be derived:
\begin{equation}
  \Bbb S \kappa=\bfb,
	\label{eq:Ax=b}
\end{equation}
where the sensitivity matrix $\Bbb S$ is defined by
\begin{equation}
  \Bbb S
	=
	\left[\begin{array}{c}
	  \vdots
		\\
  	~ \cdots ~
		\int_{D_m} \frac{\nabla v_k \cdot \nabla v_l}{V_{k,l}} - \frac{\nabla v_1 \cdot \nabla v_2}{V_{1,2}} dx
		~ \cdots ~
		\\
		\vdots
	\end{array}\right].
\end{equation}

We solve (\ref{eq:Ax=b}) using a Tikhonov regularization method as follows:
\begin{equation}
  \kappa=\left(\mathbb{S}^T\mathbb{S}+\alpha\bfI\right)^{-1}\mathbb{S}^T\bfb ,
\end{equation}
where the superscript $T$ denotes the transpose, $\alpha$ is a regularization parameter, and $\bfI$ is the identity matrix.

\section{Numerical experiments}\label{sec:numerical}
We conduct numerical simulations to test the proposed least-squares algorithm.
The boundary shape $\partial\Omega$ and electrode positions $p_1,p_2,\cdots,p_{32}$ are obtained by using 3D scanner as shown in Figure \ref{fig:fwd-humanbody} (a).
To generate the data $U_{k,l}$, we use finite element method (FEM) with mesh in the scanned domain $\Omega$.
To make internal admittivity distribution for fat, muscle, bone, and inter organs, we use a CT image and for the admittivity values of biological tissues, we use the same values as in \cite{Hasgall:2015}; see Figure \ref{fig:fwd-humanbody} (b).
\begin{figure}[h]
  \begin{center}\begin{tabular}{cc}
	  \begin{minipage}{3.8cm}\begin{tabular}{c}
  		\begin{minipage}{3.3cm}\includegraphics[width=3.3cm]{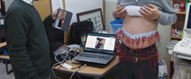}\end{minipage} \\
	  	\begin{minipage}{3.3cm}\includegraphics[width=3.3cm]{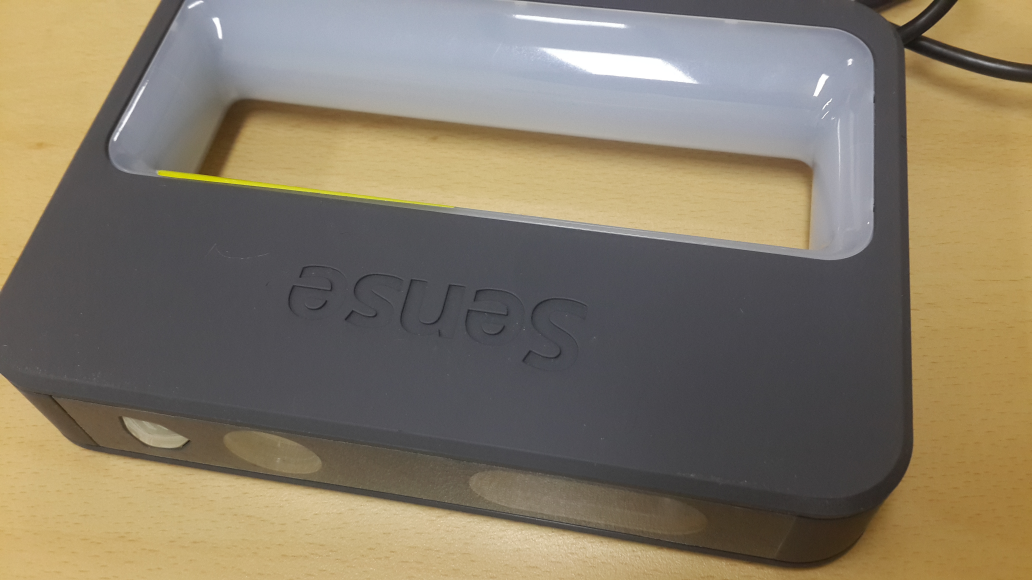}\end{minipage}
		\end{tabular}\end{minipage} &
	  \begin{minipage}{3.5cm}\includegraphics[width=3.5cm]{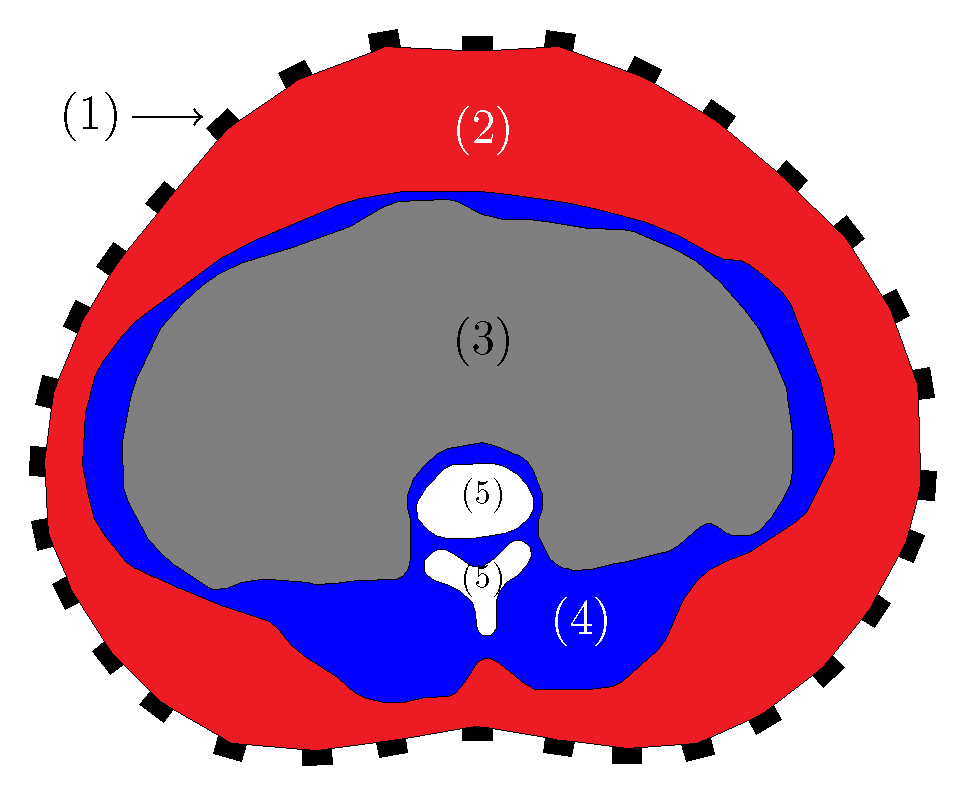}\end{minipage} ~~
	  \begin{minipage}{4cm}\tiny\begin{tabular}{l}
  		(1) Electrode: 1 \\~\\
		  (2) Subcutaneous Fat: 1/15 \\~\\
		  (3) Internal: (1/15+1/0.65)/2 \\~\\
		  (4) Muscle: 1/3 \\~\\
		  (5) Bone: 1/150 \\~\\
		  unit is [S/m]
		\end{tabular}\end{minipage} \\
		(a) & (b)
	\end{tabular}\end{center}
	\caption{(a) shows 3D scanner instrument and its usage, (b) shows used forward model and conductivity values to generate data}
	\label{fig:fwd-humanbody}
\end{figure}
As shown in Figure \ref{fig:recon-mesh} for numerical tests, we use the inject-measure set in $\mP_4$ by choosing eight electrodes depicted in red and generating the corresponding triangular mesh.

For computing geometry-free data $G_{k,l}:=V_{k,l}/U_{k,l}$, the $V_{k,l}$ is essential which can be obtained numerically by using the FEM with homogeneous admittivity $\gamma=1$ in $\Omega$.
To give conditions similar to real situations, the FEM uses optimally generated meshes corresponds to each admittivity distributions for $U_{k,l}$ and $V_{k,l}$.

We tested all inject-measure index set $\mP_1,\mP_2,\cdots,\mP_{32}$ with corresponding triangular mesh in the reconstructed region $D$ corresponds to the choice of $\mP_n$.
The image reconstruction result of applying the method using $\mP_4$ index set with 1063 triangular mesh (Figure \ref{fig:recon-mesh}) is represented in Figure \ref{fig:recon-humanbody-4methods}.  
\begin{figure}[htb]
  \begin{center}
	  \includegraphics[width=5cm]{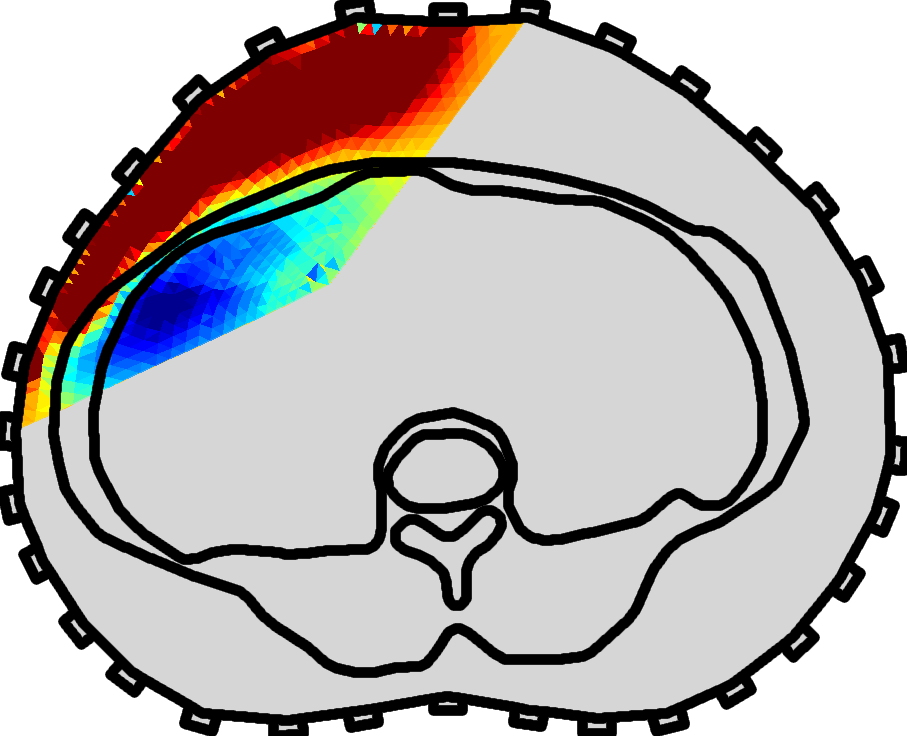}
	\end{center}
	\caption{Reconstructed image of applying the method using $\mP_4$ index set.}
	\label{fig:recon-humanbody-4methods}
\end{figure}

Now, we add Gaussian random noise with SNR 15dB to the data $U_{k,l}$.
The reconstruction results using index sets $\mP_1,\mP_2,\cdots,\mP_{32}$ of applying the linearized method with Tikhonov regularization using the noisy data are represented in Figure \ref{fig:recon-humanbody-linearized-32}.
For visual comfort, we merge the 32 result images to 1 image as shown in Figure \ref{fig:recon-humanbody-linearized-32}.
\begin{figure}[htb]
	\begin{center}
  	\includegraphics[width=13cm]{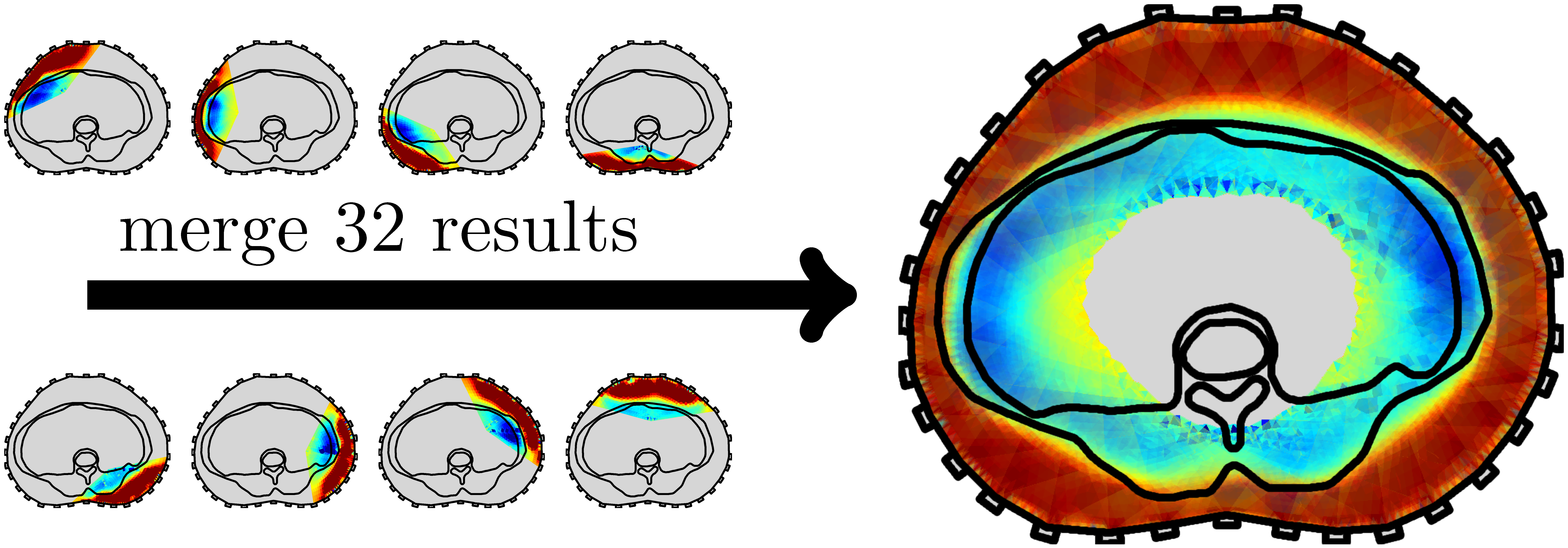}
	\end{center}
	\caption{Numerical results of applying linearized method for each $\mP_1,\mP_2,\cdots,\mP_{32}$ with noisy data.}
	\label{fig:recon-humanbody-linearized-32}
\end{figure}

It is worth pointing out that the column vectors of $\mathbb{S}$ are highly correlated.
The correlation function $c_{i,j}$ between the column vectors can be defined by
\begin{equation}
  c_{i,j}:=\frac{\bfs_i\cdot\bfs_j}{\|\bfs_i\|\|\bfs_j\|}
\end{equation}
for $i,j=1,2,\cdots,N_T$, where $\bfs_i$, $\bfs_j$ are column vectors of $\mathbb{S}$.
We compute the correlation $c_{1,j}$, $c_{2,j}$, $c_{3,j}$ for $j=1,2,\cdots,N_T$ which correspond to the mesh elements $D_1$ near the boundary, $D_2$ in the middle of $\Omega_f$, and $D_3$ far from the boundary $\partial \Omega$.
As shown in Figure \ref{fig:correlation-Smat}, the column vectors are highly correlated.  
\begin{figure}[htb]
  \begin{center}\begin{tabular}{cc}
	  \begin{minipage}{4.2cm}\includegraphics[width=4.2cm]{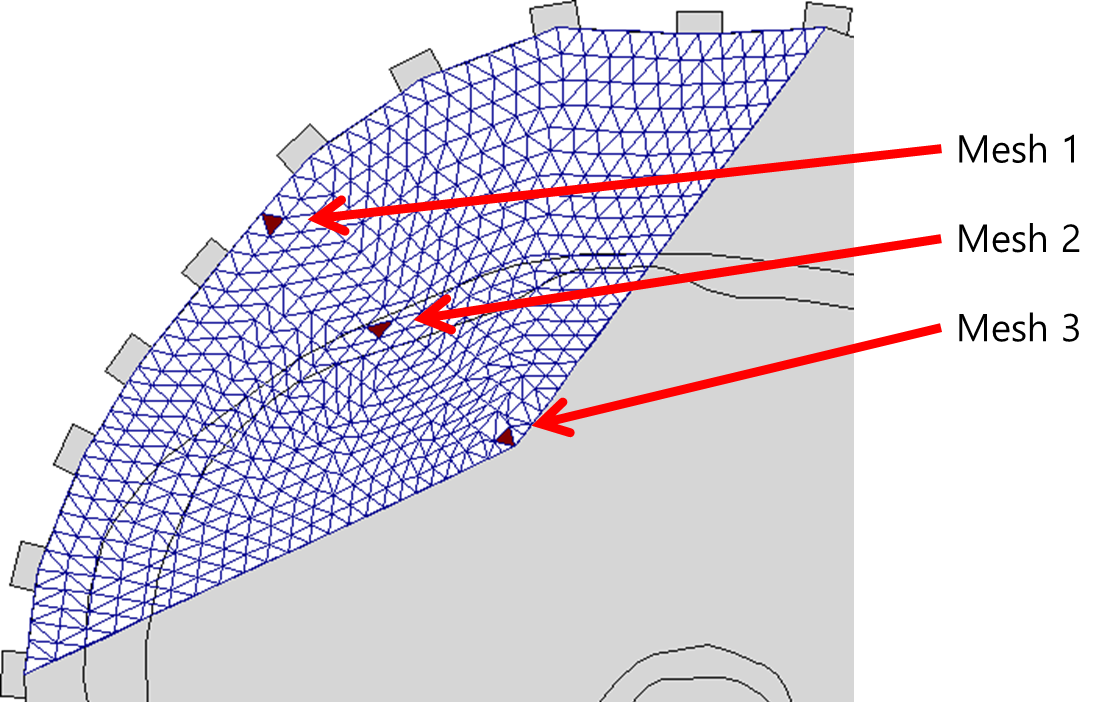}\end{minipage} &
	  \begin{minipage}{4.2cm}\includegraphics[width=4.2cm]{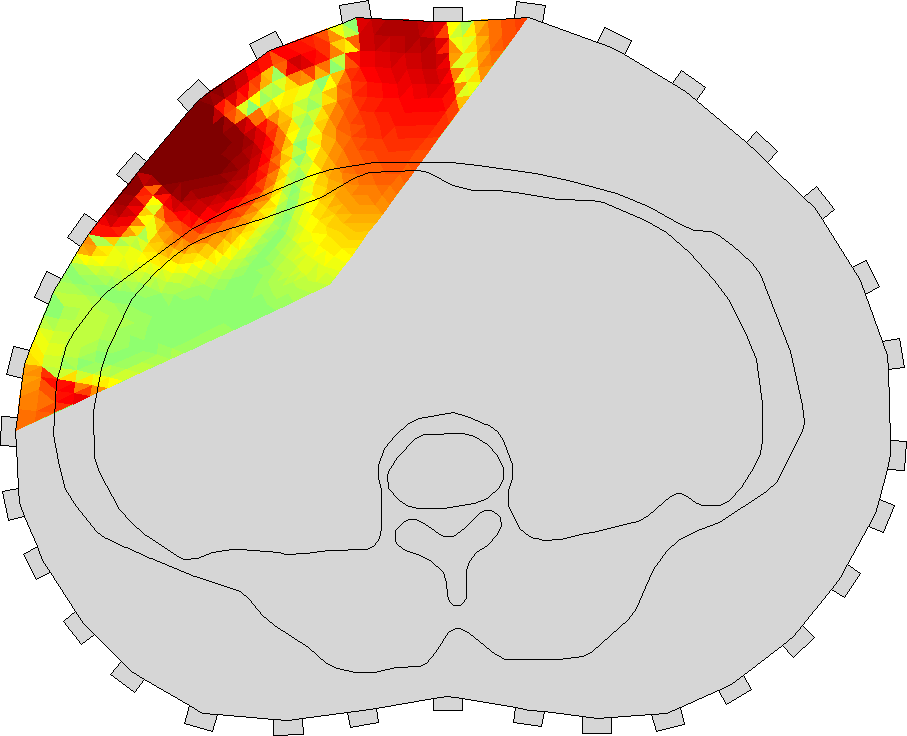}\end{minipage}
		\begin{minipage}{0.7cm}\includegraphics[width=0.7cm]{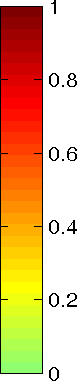}\end{minipage} \\
		(a) & (b) \\ ~ \\
	  \begin{minipage}{4.2cm}\includegraphics[width=4.2cm]{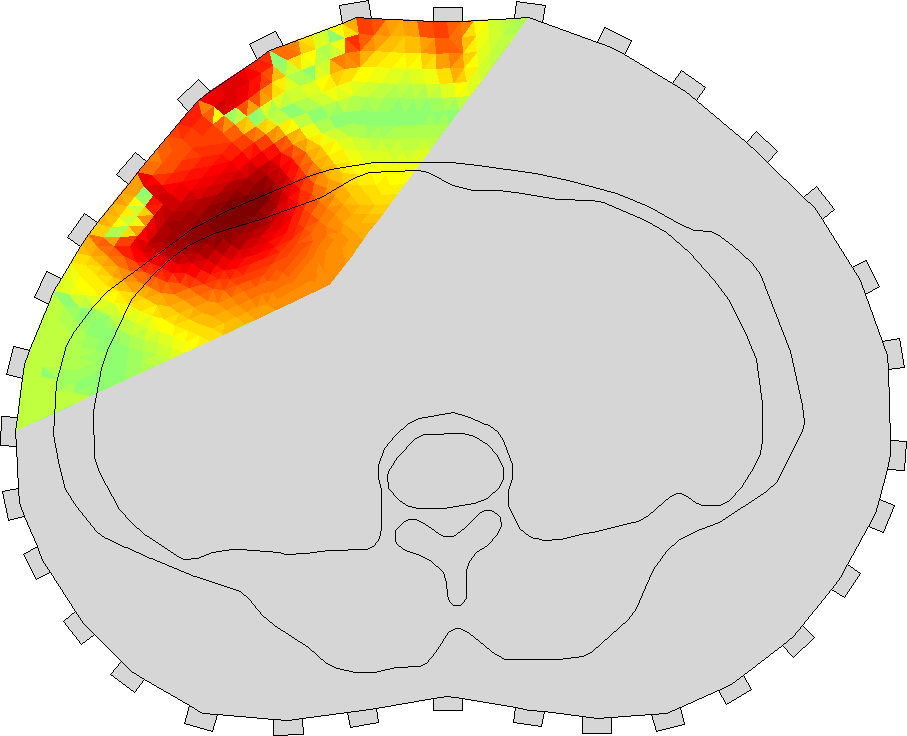}\end{minipage}
		\begin{minipage}{0.7cm}\includegraphics[width=0.7cm]{ColorBar}\end{minipage} &
	  \begin{minipage}{4.2cm}\includegraphics[width=4.2cm]{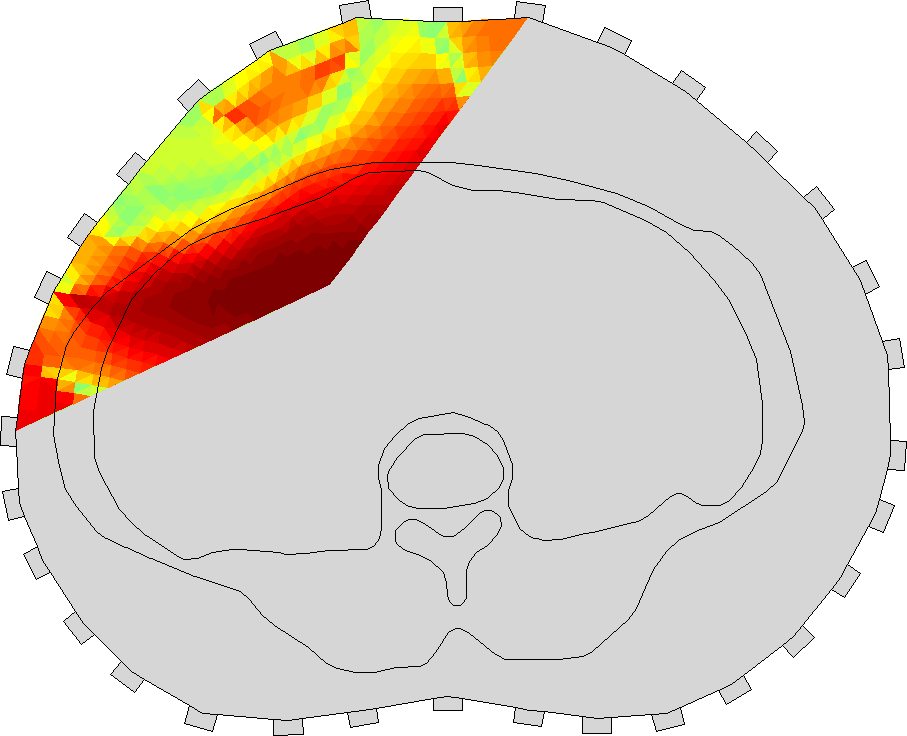}\end{minipage}
		\begin{minipage}{0.7cm}\includegraphics[width=0.7cm]{ColorBar}\end{minipage} \\
		(c) & (d)
	\end{tabular}\end{center}
	\caption{(a) represents chosen mesh elements for computing correlations. (b), (c), and (d) show correlation distributions $c_{1,j}$, $c_{2,j}$, $c_{3,j}$ with mesh 1, 2, and 3, respectively.}
	\label{fig:correlation-Smat}
\end{figure}

The success of the proposed least-squares method arises from taking advantage of the Tikhonov regularization.
The Tikhonov regularization minimizes both $|\mathbb{S}\kappa-\bfb|^2$ and $|\kappa|^2$ as follows:
$$
  \kappa=\left(\mathbb{S}^T\mathbb{S} +\alpha\bfI\right)^{-1}\mathbb{S}^T\bfb
	\quad\Leftrightarrow\quad
	\kappa=\mbox{\rm arg}\min_\kappa|\mathbb{S}\kappa-\bfb|^2+\alpha|\kappa|^2 .
$$
The component of $\kappa=[\kappa_1,\kappa_2,\cdots,\kappa_{N_T}]^T$ can be considered as a coefficient of the linear combination of column vectors $\bfs_n$ to generate $\bfb$.
When minimizing $|\mathbb{S}\kappa-\bfb|^2$, high correlation of $\bfs_n$ causes uncertainty in finding $\kappa$. 
However, minimizing $|\kappa|^2$ can compensate the mismatch of $\kappa_n$ caused by high correlations between the $\bfs_n$ for $n=1,2,\cdots,N_T$, since $\kappa_n$ and $\kappa_m$ tend to have a similar value when $\bfs_n$ and $\bfs_m$ are highly correlated.

It is also worth emphasizing that other imaging algorithms such as MUSIC \cite{hanke,hanke2,Cheney:2001} can not perform well because precisely of the high correlation of column vectors of the sensitivity matrix $\mathbb{S}$.

\section{Conclusions}
In this work, static EIT image reconstruction algorithm of human abdomen for identifying subcutaneous fat region is developed.
The proposed depth-based reconstruction method relies on Theorem \ref{thm:main} which shows that subcutaneous fat influence in the data can be eliminated by using a reference-like data and geometry information (domain shape and electrode configuration) to overcome a fundamental drawback in static EIT; lack of reference data for handling the forward modeling errors.
We suggest a linearized method with Tikhonov regularization which uses the subcutaneous influence eliminated data with a specially chosen current pattern.
Numerical simulations show that the reconstruction result of identifying the subcutaneous fat region is quite satisfactory.
The suggested way of eliminating influence of homogeneous background admittivity can be applied in other static EIT area, for instance, ground detection.
The knowledge of subcutaneous fat region can be a useful information of developing an algorithm of estimating visceral fat occupation. For clinical use,  estimating visceral fat occupation is required to provide useful information in abdominal obesity.

\appendix
\section{Complete electrode model} \label{sec:appendix}
The following result holds.
\begin{lemma}
Let $\Omega$ be a lower half-space of $\mathbb{R}^3$.
Let $\mE^h_+$ and $\mE^h_-$ be circular electrodes centered at points $p_+$ and $p_-$ on $\partial\Omega$, respectively, with radius $h>0$.
Let {$\widetilde{u}^h\in H^1_{loc}(\Omega)$} satisfy
\begin{equation}
  \left\{\begin{array}{l}
    \Delta \widetilde{u}^h = 0 \quad \text{in}~ \Omega , \\
    \left.\widetilde{u}^h+z_\pm\frac{\partial \widetilde{u}^h}{\partial n}\right|_{\mE^h_\pm} = \widetilde{P}^h_\pm ,\\
	\int_{\mE^h_+}\frac{\partial \widetilde{u}^h}{\partial n} ~ds = 1 =-\int_{\mE^h_-}\frac{\partial \widetilde{u}^h}{\partial n} ~ds ,\\
	\frac{\partial \widetilde{u}^h}{\partial n} = 0 \quad \mbox{on}~\partial\Omega\setminus\left(\mE^h_+\cup\mE^h_-\right) ,\\
    \lim_{|x|\rightarrow\infty}\widetilde{u}^h(x)=0, 
  \end{array}\right.
  \label{eq:CEMu}
\end{equation}
where $z_\pm$ and $\widetilde{P}^h_\pm$ are constants.
Define $v(x):=\left(\Phi(x,p_+)-\Phi(x,p_-)\right)$.
For given $R>0$, there exists a constant $C$ such that, for $2h<R$, 
\begin{equation}
  \|\widetilde{u}^h - v\|_{H^1(\Xi_R)} \le C h ,
\end{equation}
where $\Xi_R := \{x \in \Omega ~|~ |x-p_+|>R,~ |x-p_-|>R\}$.
\label{lem:CEMisPEM}
\end{lemma}
\begin{proof}
Let $v^h$ be the solution of 
\begin{equation}
  \left\{\begin{array}{l}
    \Delta v^h = 0 \quad\mbox{in}~\Omega , \\
    \frac{\partial}{\partial n}v^h = -\frac{1}{2\pi h^2}\left(\chi_{\mE^h_+}-\chi_{\mE^h_-}\right) \quad\mbox{on}~\partial\Omega , \\
    \lim_{|x|\rightarrow\infty}v^h(x)=0 .
  \end{array}\right.
\end{equation}
Then $v^h$ can be represented by
\begin{equation}
  v^h(x)=\frac{1}{\pi h^2}\mS\left[\chi_{\mE^h_+}-\chi_{\mE^h_-}\right](x)
  \quad\mbox{for}~x\in\Omega .
\end{equation}
For $x\in\Xi_R$, we get
\begin{equation}
  v^h(x)-v(x) = \frac{1}{\pi h^2}\int_{\mE^h_+}\Phi(x,y)-\Phi(x,p_+)ds_y-\frac{1}{\pi h^2}\int_{\mE^h_-}\Phi(x,y)-\Phi(x,p_-)ds_y .
\end{equation}
It follows from mean-value theorem that, for $x\in\Xi_R$,
\begin{eqnarray}
  |v^h(x)-v(x)|
	&\le&
	h \left(\sup_{y\in\mE_+^h}|\nabla\Phi(x,y)|+\sup_{y\in\mE_-^h}|\nabla\Phi(x,y)|\right)
	\\&\le&
	\frac{h}{4\pi(|x-p_+|-h)^2}+\frac{h}{4\pi(|x-p_-|-h)^2} .
\end{eqnarray}
Since $2h<|x-p_\pm|$ for $x\in\Xi_R$, we have
\begin{equation}
  |v^h(x)-v(x)| \le  \frac{h}{\pi|x-p_+|^2}+\frac{h}{\pi|x-p_-|^2} .
\end{equation}
Therefore
\begin{equation}
  \|v^h-v\|_{H^1(\Xi_R)} \le Ch\left(\frac{1}{\sqrt{R}}+\frac{1}{R\sqrt{R}}\right) .
\end{equation}

For $x\in\Omega$, $\widetilde{u}^h$ and $v^h$ can be represented by
\begin{eqnarray*}
  \widetilde{u}^h(x) &=&
  -2\mS\left[\frac{\partial \widetilde{u}^h}{\partial n}\chi_{\mE^h_+}\right](x)-2\mS\left[\frac{\partial \widetilde{u}^h}{\partial n}\chi_{\mE^h_-}\right](x) ,
  \\
  v^h(x) &=&
  -2\mS\left[\frac{\partial v^h}{\partial n}\chi_{\mE^h_+}\right](x)-2\mS\left[\frac{\partial v^h}{\partial n}\chi_{\mE^h_-}\right](x).
\end{eqnarray*}
Then, for $x\in\Xi_R$,
\begin{equation}
  \widetilde{u}^h(x)-v^h(x)
  =
  -2\int_{\mE^h_+} \Phi(x,y)\eta^h_+(y) ~ds_y
  -2\int_{\mE^h_-} \Phi(x,y)\eta^h_-(y) ~ds_y , 
\end{equation}
where $\eta^h_\pm:=\frac{\partial \widetilde{u}^h}{\partial n}-\frac{\partial v^h}{\partial n}$ on $\mE^h_\pm$.
Since $\int_{\mE_\pm}\eta^h_\pm(y) ~ds_y=0$, 
\begin{eqnarray*}
  \widetilde{u}^h(x)-v^h(x)
  &=&
  -2\int_{\mE^h_+} (\Phi(x,y)-\Phi(x,p_+))\eta^h_+(y) ~ds_y
  \\ &~&
  -2\int_{\mE^h_-} (\Phi(x,y)-\phi(x,p_-))\eta^h_-(y) ~ds_y .
\end{eqnarray*}
It follows from mean-value theorem that for $x\in\Xi_R$
{\small \begin{equation}
  |\widetilde{u}^h(x)-v^h(x)|
  \le
  2h\sup_{y\in\mE_+^h}|\nabla\Phi(x,y)|\int_{\mE^h_+}|\eta^h_+(y)| ~ds_y
  +
  2h\sup_{y\in\mE_-^h}|\nabla\Phi(x,y)|\int_{\mE^h_-}|\eta^h_-(y)| ~ds_y .
\end{equation}  }
A similar argument shows that
\begin{equation}
  \|\widetilde{u}^h-v^h\|_{H^1(\Xi_R)} \le C'h\left(\frac{1}{\sqrt{R}}+\frac{1}{R\sqrt{R}}\right)\left(\|\eta^h_+\|_{L^1(\mE^h_+)}+\|\eta^h_-\|_{L^1(\mE^h_-)}\right) .
\end{equation}
\end{proof}  

We can extend the above lemma for the half-space case to bounded domains as follows.
\begin{theorem}
Let $\Omega$ be a simply connected domain in $\mathbb{R}^3$.
Let $\mE^h_+$ and $\mE^h_-$ be circular electrodes centered at points $p_+$ and $p_-$ on $\partial\Omega$, respectively, with radius $h>0$.
Let $\widetilde{u}^h\in H^1(\Omega)$ satisfy
\begin{equation}
  \left\{\begin{array}{l}
    \Delta \widetilde{u}^h = 0 \quad \text{in}~ \Omega ,\\
    \left.\widetilde{u}^h+z_\pm\frac{\partial \widetilde{u}^h}{\partial n}\right|_{\mE^h_\pm} = \widetilde{P}^h_\pm ,\\
	\int_{\mE^h_+}\frac{\partial \widetilde{u}^h}{\partial n} ~ds = I =-\int_{\mE^h_-}\frac{\partial \widetilde{u}^h}{\partial n} ~ds ,\\
	\frac{\partial \widetilde{u}^h}{\partial n} = 0 \quad \mbox{on}~\partial\Omega\setminus\left(\mE^h_+\cup\mE^h_-\right) , \\
    \int_{\partial\Omega}\widetilde{u}^h ~ds=0 , 
  \end{array}\right.
  \label{eq:CEMu2}
\end{equation}
where $z_\pm$ and $\widetilde{P}^h_\pm$ are constants.
Let $v\in H^1(\Omega)$ satisfy
\begin{equation}
  \left\{\begin{array}{l}
    \Delta v = 0 \quad \text{in}~ \Omega, \\
    \frac{\partial}{\partial n}v(x) = I (\delta(x-p_+) - \delta(x-p_-)) \quad \text{for}~ x\in\partial\Omega ,\\
		\int_{\partial\Omega} v ~ds = 0 .
  \end{array}\right.
  \label{eq:PEMv2}
\end{equation}
For given $R>0$, there exists a constant $C$ such that, for $2h<R$, 
\begin{equation}
  \|\widetilde{u}^h - v\|_{H^1(\Xi_R)} \le C h , 
\end{equation}
where $\Xi_R := \{x \in \Omega ~|~ |x-p_+|>R,~ |x-p_-|>R\}$.
\label{thm:CEMisPEM}
\end{theorem}

\begin{proof}Let $v^h$ be the solution of the following equation
\begin{equation}
  \left\{\begin{array}{l}
    \Delta v^h = 0 \quad\mbox{in}~\Omega , \\
    \frac{\partial}{\partial n}v^h = \frac{I}{\pi h^2}\left(\chi_{\mE^h_+}-\chi_{\mE^h_-}\right) \quad\mbox{on}~\partial\Omega , \\
    \int_{\partial\Omega}v^h ds=0 .
  \end{array}\right.
\end{equation}
Then $v^h$ and $v$ can be represented by
\begin{equation}
  v^h(x)=\frac{I}{\pi h^2}\int_{\mE^h_+}\mN_1(x,y)ds_y - \frac{I}{\pi h^2}\int_{\mE^h_-}\mN_1(x,y)ds_y
  \quad\mbox{for}~x\in\Omega,
\end{equation}
and 
\begin{equation}
  v(x)=I \left( \mN_1(x,p_{+}) - \mN_1(x,p_{-}) \right)
  \quad\mbox{for}~x\in\Omega.
\end{equation}
Hence, for $x\in\Xi_R$,
\begin{equation}
  v^h(x)-v(x) = \frac{I}{\pi h^2}\int_{\mE^h_+}\mN_1(x,y)-\mN_1(x,p_{+})ds_y-\frac{I}{\pi h^2}\int_{\mE^h_-}\mN_1(x,y)-\mN_1(x,p_{-})ds_y .
\end{equation}
Therefore, from mean-value theorem it follows that, for $x\in\Xi_R$, 
\begin{equation*}
  |v^h(x)-v(x)|	\le  I h \left(\sup_{y\in\mE_+^h}|\nabla\mN_1(x,y)|+\sup_{y\in\mE_-^h}|\nabla\mN_1(x,y)|\right).
\end{equation*}
According to the decay estimation for the Neumann function in \cite{Widman:1982} and the fact that $2h<|x-p_\pm|$ for $x\in\Xi_R$, there exists a positive constant $C$ such that
\begin{eqnarray*}
  |v^h(x)-v(x)| &\le& \frac{C I h}{(|x-p_+|-h)^2}+\frac{C I h}{(|x-p_-|-h)^2} \\
  &\le& \frac{4C I h}{|x-p_+|^2}+\frac{4C I h}{|x-p_-|^2} .
\end{eqnarray*}
Consequently, 
\begin{equation}
  \|v^h-v\|_{H^1(\Xi_R)} \le \tilde{C} I h \left(\frac{1}{\sqrt{R}}+\frac{1}{R\sqrt{R}}\right).
\end{equation}

The functions $\widetilde{u}^h$ and $v^h$ can be represented by
\begin{eqnarray*}
  \widetilde{u}^h(x) &=&
  \int_{\mE^h_+}\mN_1(x,y) \frac{\partial \widetilde{u}^h}{\partial n_y}(y) ~ds_y + \int_{\mE^h_-}\mN_1(x,y) \frac{\partial \widetilde{u}^h}{\partial n_y}(y) ~ds_y 
  \quad\mbox{for}~x\in\Omega,
  \\
  v^h(x) &=&
  \int_{\mE^h_+}\mN_1(x,y) \frac{\partial v^h}{\partial n_y}(y) ~ds_y + \int_{\mE^h_-}\mN_1(x,y) \frac{\partial v^h}{\partial n_y}(y) ~ds_y
  \quad\mbox{for}~x\in\Omega.
\end{eqnarray*}
Then, for $x\in\Xi_R$,
\begin{equation}
  \widetilde{u}^h(x)-v^h(x)
  =
  \int_{\mE^h_+}\mN_1(x,y) \eta^h_+(y) ~ds_y + \int_{\mE^h_-}\mN_1(x,y) \eta^h_-(y) ~ds_y  , 
\end{equation}
where $\eta^h_\pm(y):=\frac{\partial \widetilde{u}^h}{\partial n_y}(y)-\frac{\partial v^h}{\partial n_y}(y)$ on $\mE^h_\pm$.
Since $\int_{\mE_\pm}\eta^h_\pm(y) ~ds_y=0$,
\begin{eqnarray*}
  \widetilde{u}^h(x)-v^h(x)
  &=&
  \int_{\mE^h_+}( \mN_1(x,y) - \mN_1(x,p_+)) \eta^h_+(y) ~ds_y   
  \\ &~&
  + \int_{\mE^h_-}( \mN_1(x,y) - \mN_1(x,p_-)) \eta^h_-(y) ~ds_y .
\end{eqnarray*}
Again, from the mean-value theorem it follows that, for $x\in\Xi_R$, 
{\small \begin{equation}
  |\widetilde{u}^h(x)-v^h(x)|
  \le
  h\sup_{y\in\mE_+^h}|\nabla\mN_1(x,y)|\int_{\mE^h_+}|\eta^h_+(y)| ~ds_y
  +
  h\sup_{y\in\mE_-^h}|\nabla\mN_1(x,y)|\int_{\mE^h_-}|\eta^h_-(y)| ~ds_y .
\end{equation}  }
A similar argument shows that
\begin{equation}
  \|\widetilde{u}^h-v^h\|_{H^1(\Xi_R)} \le C'h\left(\frac{1}{\sqrt{R}}+\frac{1}{R\sqrt{R}}\right)\left(\|\eta^h_+\|_{L^1(\mE^h_+)}+\|\eta^h_-\|_{L^1(\mE^h_-)}\right).
\end{equation}
Therefore
\begin{equation}
  \|\widetilde{u}^h-v\|_{H^1(\Xi_R)} \le \|\widetilde{u}^h-v^h\|_{H^1(\Xi_R)} + \|v^h-v\|_{H^1(\Xi_R)} 
  \le Ch,
\end{equation}
which completes the proof of the theorem. 
\end{proof}


\end{document}